\documentclass[12pt, reqno]{amsart}

\usepackage{amsmath,amscd, float,rotating}
\usepackage{pb-diagram}
\usepackage{latexsym}
\usepackage{amsfonts,color}
\usepackage{amsmath}
\usepackage{amssymb}
\usepackage{txfonts}

\newcommand{\CC}{\mathbb{C}}
\newcommand{\R}{\mathbb{R}}

\newcommand{\N}{\mathbb{N}}

\newcommand{\PSH}{\mathrm{PSH}}
\newcommand{\SRF}{\mathrm{SRF}}
\newcommand{\SF}{\mathrm{SF}}
\newcommand{\Ric}{\mathrm{Ric}}
\newcommand{\Rm}{\mathrm{Rm}}
\newcommand{\Kod}{\textrm{Kod}}

\newcommand{\de}{\partial}

\newcommand{\dbar}{\overline{\partial}}
\newcommand{\ddt}{\frac{\partial}{\partial t}}

\newcommand{\ddbar}{\sqrt{-1} \partial \overline{\partial}}

\newcommand{\tr}[2]{\mathrm{tr}_{#1}{#2}}
\newcommand{\vp}{\varphi}

\newcommand{\om}{\omega}
\newcommand{\Om}{\Omega}
\newcommand{\ti}[1]{\tilde{#1}}

\newcommand{\of}{\omega_{\textrm{SRF}}}

\newtheorem{theorem}{Theorem}[section]
\newtheorem{lemma}[theorem]{Lemma}
\newtheorem{corollary}[theorem]{Corollary}
\newtheorem{proposition}[theorem]{Proposition}

\numberwithin{equation}{section}

\theoremstyle{definition}
\newtheorem{remark}[theorem]{Remark}

\theoremstyle{definition}

\begin{document}

\title{A ``boundedness implies convergence'' principle and its applications to collapsing estimates in K\"ahler geometry}

\begin{abstract}
We establish a general ``boundedness implies convergence'' principle for a family of  evolving Riemannian metrics. We then apply this principle to collapsing Calabi-Yau metrics and normalized K\"ahler-Ricci flows on torus fibered minimal models to obtain convergence results.
\end{abstract}
\author{Wangjian Jian}
\address{School of Mathematical Sciences, Peking University, Beijing, China 100871}
\email{1401110008@pku.edu.cn}
\author{Yalong Shi}
\address{Department of Mathematics, Nanjing University, Nanjing, China 210093}
\email{shiyl@nju.edu.cn}
\maketitle
\thispagestyle{empty}
\markboth{Boundedness implies convergence}{Wangjian Jian and Yalong Shi}

\tableofcontents
%%%%%%%%%%%%%%%%%%%%%%%%%%%%%%%%%%%%%%%%%%%%%%%%%%%%%%%%%%%%%%%%%%%%%%%%%%%%%%%%%%%%%%%%%%%%%%%%%%%%%%%%%%%%%%%%%%%%%%%%%%%
%%%%%%%%%%%%%%%%%%%%%%%%%%%%%%%%%%%%%%%%%%%%%%%%%%%%%%%%%%%%%%%%%%%%%%%%%%%%%%%%%%%%%%%%%%%%%%%%%%%%%%%%%%%%%%%%%%%%%%%%%%%

\section{Introduction}\label{intro}

 In differential geometry, we usually try to find ``canonical'' geometric objects, e.g. Einstein metrics, minimal surfaces, etc., through a deformation process. Starting with a given object, we deform it through a suitable geometric flow or continuity path to get the desired metric or submanifold. To obtain convergence, we need a priori bounds for the higher order covariant derivatives of the evolving functions, metrics or curvature tensors. In general, such bounds are invalid and singularities may form. However, if we have curvature bounds or non-collapsing condition, general theories from Riemannian geometry and PDE usually gave us a lot of information about the formation of singularities and in the end we still get good geometric results. On the other hand, there are some important cases in which we lose curvature bounds and the metrics may collapse to lower dimensional spaces. In such cases, a priori bounds play a more decisive role. It is still interesting to know whether certain tensor fields converge to some degenerate or singular one.

In this paper, we prove a simple but useful ``Boundedness Implies Convergence" lemma, which says that in certain situation if we have the convergence of an evolving tensor field together with uniform estimates of its covariant derivatives, then we automatically have the convergence of its covariant derivatives. In spirit, this is similar to the simple fact in calculus that if a smooth function on $\R^n$ has bounded partial derivatives of all orders and if the function converges to 0 at infinity, then all its partial derivatives converge to 0 at infinity. Indeed, if we have uniform equivalence of metrics, this follows easily from interpolation type inequalities. Instead, here we assume the existence of good cut-off functions. In applications, such cut-off functions usually arise as the pull back of cut-off functions from the base manifold of a fibration. To be precise, we have:

\begin{lemma}[The ``Boundedness Implies Convergence'' Principle]\label{Thm: BIC principle}
Let $X$ be an n-dimension Riemannian manifold (not necessarily to be compact or complete) and $U$ be an open subset. Let $\ti{g}(t)$ be a family of Riemannian metrics on $X$, $t\in \R$ and let $\eta(t)$ be a family of smooth functions or general tensor fields on $X$, satisfying the following conditions:

There exists positive constants $A_1, A_2, \dots$, and a positive function $h_0(t)$ which tends to zero as $t\to \infty$ such that
\begin{enumerate}
\item [(A)] ~$\|\eta(t)\|_{C^0(U,\ti{g}(t))}\leq h_0(t).$
\item [(B)] ~$\|\eta(t)\|_{C^k(U,\ti{g}(t))}\leq A_k, ~for~ $k=1,2,\dots$~. $
\item [(C)] ~For any compact subset $K\subset\subset U$, there exists  smooth cut-off function $\rho$ with compact support $\hat{K}\subset\subset U$ such that $0\leq\rho\leq 1$, and  $\rho\equiv 1$ in a neighborhood of  $K$, satisfying
\begin{equation}\label{Equ: cutoff bound}
\left|\nabla{\rho}\right|_{\ti{g}(t)}^2+|\Delta_{\ti{g}(t)}\rho|\leq B_{K}.
\end{equation}
on $\hat{K}\times [0,\infty)$, for some constant $B_{K}$ independent of $t$ (but may depend on the geometry of $K$).
\end{enumerate}

Then we have: For any compact subset $K\subset\subset U$ the estimates
\begin{equation}\label{Equ: convergence in general}
\|\eta(t)\|_{C^k(K,\ti{g}(t))}\leq h_{K,k}(t).
\end{equation}
where $h_{K,k}(t)$ are positive functions which tend to zero as $t\to \infty$, depenging on the constants $A_0$, $A_1$, $\dots$, $A_{k+2}$, $B_K$ and the function $h(t)$.
\end{lemma}

\begin{remark}\begin{enumerate}
\item Lemma \ref{Thm: BIC principle} also holds true for discrete sequences of  metrics and tensors. This is easily seen from the proof in section \ref{sec: pf of BIC}.
\item In Condition $(A)$ if $h_0(t)$ is of the form $Ce^{-ct}$ for some positive constants $C<\infty$ and $c>0$, then the functions $h_{K,k}(t)$ can be chosen to be $C_{K,k}e^{-c_{K,k} t}$ for some constants $C_{K,k}, c_{K,k}$, i.e. covariant derivatives of $\eta(t)$ also decay exponentially.
\item In  Lemma \ref{Thm: BIC principle} if we only have Condition $(B)$ for $1\leq k\leq N+2$, then we still have the estimate \eqref{Equ: convergence in general} for $1\leq k\leq N$.
\end{enumerate}
\end{remark}

Note that in Lemma \ref{Thm: BIC principle} we do not require the metrics to be uniformly non-degenerate or have bounded curvature and do not require the metrics and tensors satisfy differential equations. So it applies even when the metrics collapse to lower dimensional spaces without curvature bounds. In particular, we shall apply this principle to two collapsing problems in K\"ahler geometry. We hope this simple ``BIC principle'' may find other applications in geometric analysis, especially in collapsing problems.\\

The first application in this paper is on Calabi-Yau degenerations.

Given a Calabi-Yau manifold $X$ with a holomorphic fiber space structure (i.e. there is a holomorphic surjection $f$ onto a lower dimensional variety, which is a holomorphic submersion outside a subvariety), the Calabi-Yau theorem \cite{Yau} assures the existence of unique Ricci-flat K\"ahler metrics on the total space in every K\"ahler classes. Now let the K\"ahler classes approach the pull back of a K\"ahler class from the base, then the volume of these Ricci-flat metrics go to zero. One would like to understand the asymptotic behaviors of these degenerating metrics.

This problem has been much studied in recent years, starting from the pioneering work of Gross-Wilson \cite{GW} on elliptic $K3$ surfaces fibered over the 2-sphere, to more recent works in general dimensions by Tosatti \cite{To}, Tosatti-Zhang \cite{TZ, TZ2}, Gross-Tosatti-Zhang \cite{GTZ, GTZ2}, Hein-Tosatti \cite{HT}, Tosatti-Weinkove-Yang \cite{TWY}, Li \cite{Li}, and elsewhere. From these works, we know that the Ricci-flat metrics collapse, in some sense, to the pullback of a canonical K\"ahler metric on the base, uniformly on compact sets away from the singular fibers.

In this general setting, a $C^0_{\rm loc}$ estimate was proved by Tosatti-Weinkove-Yang in \cite{TWY}, and this estimate can be improved to $C^\infty_{\rm loc}$ estimate when the smooth fibers are tori or finite \'etale quotients of tori by Gross-Tosatti-Zhang \cite{GTZ} and Hein-Tosatti \cite{HT}. Certain components of the first order derivatives were bounded by Tosatti \cite{To} and Tosatti-Weinkove-Yang in \cite{TWY}. An even stronger partial estimate was proved by Tosatti-Zhang in \cite{TZ}: the restriction of $e^t\om(t)$ to $f^{-1}(U)$ converges in the pointed $C^\infty$ Cheeger-Gromov topology to the product of a flat $\mathbb{C}^m$ with a fiber equipped with Ricci-flat metric.

In \cite{HT2}, with a systematic use of iterated blow-up-and-contradiction type arguments, Hein-Tosatti substantially improved the estimate to $C^\infty_{{\rm loc}}$ if the regular fibers are pairwise bi-holomorphic to each other. In the general fibration case, they can improve the $C^0$ convergence of \cite{TWY} to $C^\alpha$ convergence. In a more recent preprint \cite{STZ}, Song-Tian-Zhang proved the uniform diameter bound and the Gromov-Hausdorff convergence of this family of collapsing metrics.

In this paper, we derive from Hein-Tosatti's $C^\infty_{{\rm loc}}$ estimate the corresponding convergence results, and such estimates also imply the $C^\infty_{{\rm loc}}$ asymptotic behavior of the curvature tensor.

To state our results, let $X$ be a compact K\"ahler $(n+m)$-manifold with $c_1(X)=0$ in $H^2(X,\mathbb{R})$ (i.e. a Calabi-Yau manifold), and let $\om_X$ be a Ricci-flat K\"ahler metric on $X$. Suppose that we have a holomorphic map $f:X\to Z$ with connected fibers, where $(Z,\om_Z)$ is a compact K\"ahler manifold, with image $B=f(X)\subset Z$ an irreducible normal subvariety of $Z$ of dimension $m>0$. Then the induced surjective map $f:X\to B$ is a fiber space, and if $S'\subset B$ denotes the singular set of $B$ together with the set of critical values of $f$, and $S=f^{-1}(S')$, then $S'$ is a proper analytic subvariety of $B$, $S$ is a proper analytic subvariety of $X$, and $f:X\backslash S\to B\backslash S'$ is a submersion between smooth manifolds. The fibers $X_b$ for $b \in B \backslash S'$ are Calabi-Yau $n$-folds. Write $\chi=f^*\omega_Z$, which is a smooth nonnegative $(1,1)$ form on $X$, and we will also write $\chi$ for the restriction of $\omega_Z$ to $B\backslash S'$.  

Let $\om_B$ be the current in $[\chi]$ that is smooth and positive on $B\backslash S'$  satisfying
$$\Ric(\om_B)=\om_{{\rm WP}},$$
where $\om_{{\rm WP}}$ is a smooth semipositive Weil-Petersson form on $B\backslash S'$. Its construction will be briefly recalled in section \ref{sec:CY degeneration 2}.  We also have the semi-Ricci flat metric $\of$ on $X\backslash S$, such that for each $b \in B \backslash S'$ , $\of|_{X_b}$ is the  unique Ricci-flat metric on $X_b$ in the K\"ahler class $[\om_X|_{X_b}]$. Define the reference metrics $\ti{\om}_t$ on the regular part $X\backslash S$ by 
\begin{align}
\ti{\om}_t=\om_B+e^{-t}\of.
\end{align}
Let $\om(t)$ be the unique Ricci-flat metric in $[\ti{\om}_t]= [\chi] + e^{-t} [\om_X]$. Then Tosatti-Weinkove-Yang proved in \cite{TWY} that $\|\om(t)-\ti{\om}_t\|_{C^0(K,\ti{\om}_t)} \rightarrow 0, \quad \textrm{as } t\rightarrow \infty$ on any compact set $K\subset X\setminus S$.
Denote the (1,3)-curvature tensor of a K\"ahler metric $\om$ (associated to the Riemannian metric $g$) by $R^\sharp(\om)$, i.e.
\begin{equation}\label{Equ: definition of R-sharp tensor}
{R^\sharp(\om)_{i\bar{j}k}}^l=g^{l\bar{q}}R(\om)_{i\bar{j}k\bar{q}}.
\end{equation}
If all the regular fibers are bi-holomorphic to each other, then $\om_{\textrm{SRF,b}}$ is independent of the base regular point $b$ (see Equation \ref{Equ: identity of reference metric on Y}), and at this time we define the K\"ahler metric $\om_Y$ on $Y$ by
$$\om_Y=\om_{\textrm{SRF,b}},$$
for any $b\in B \backslash S'$. Based on the higher estimates of Hein-Tosatti \cite{HT2}, we have:

\begin{theorem}\label{Thm: convergence of higher order estimate for CY}
Assume all the regular fibers are bi-holomorphic to a fixed Calabi-Yau manifold $Y$. Let $U\subset B\setminus S'$ be an open set such that the fibration is holomorphically trivial over $U$. Identify $f^{-1}(U)$ with $U\times Y$.  Define another reference metrics $\ti{\om}(t)$ on $U\times Y$ by
\begin{align}
\ti{\om}(t)=\om_B+e^{-t}\om_Y.
\end{align}
Then for each compact set $K \subset U$, for any $k \in \N$, we have
\begin{equation}\label{Equ: convergence of higher order estimate for CY}
\|\om(t)-\ti{\om}(t)\|_{C^k(U\times Y,\ti{\om}(t))}\leq h_{K,k}(t),
\end{equation}
and
\begin{equation}\label{Equ: convergence of Rm in CY case}
\|R^\sharp(\om(t))-R^\sharp(\ti{\om}(t))\|_{C^k(K\times Y,\ti{\om}(t))}\leq h_{K,k}(t),
\end{equation}
where $h_{K,k}(t)$ are positive functions which tends to zero as $t\to \infty$, depenging only on $k$ and the domain $K$. 
\end{theorem}

Note that since we do not have curvature bounds, we can not derive the decay estimate of $\|Rm(\om(t))-Rm(\ti{\om}(t))\|_{C^k}$ from (\ref{Equ: convergence of Rm in CY case}).\\

Another special case is when the smooth fibers $X_b$ are all complex torus by a holomorphic free action of a finite group, but we allow the complex structure to change. By \cite{GTZ, HT, TZ}, we have $C^\infty$ estimates as well as local curvature bounds on $X\backslash S$. We can apply Lemma \ref{Thm: BIC principle} to obtain:

\begin{theorem}\label{Thm: convergence Rm of CCY with torus fiber}
Assume that for some $b\in B\backslash S'$ the fiber $X_b=f^{-1}(b)$ is bi-holomorphic to a finite quotient of a torus. Let $K\subset X\backslash S$ be any compact subset. Then we have
\begin{equation}\label{Equ: convergence of CCY with torus fiber}
\|\om(t)-\ti\om_t\|_{C^k(K,\ti\om_t)}\leq h_{K,k}(t).
\end{equation}
and
\begin{equation}\label{Equ: convergence Rm of CCY with torus fiber}
\|\Rm(\om(t))-\Rm(\ti\om_t)\|_{C^k(K,\ti\om_t)}\leq h_{K,k+2}(t).
\end{equation}
where $h_{K,k}(t)$ are positive functions which tends to zero as $t\to \infty$, depending only on $k$ and the domain $K$. In particular, when $S=\emptyset$, the estimates are globally true and each $h_k(t)$ is of exponential fast decay.
\end{theorem}

%%%%%%%%%%%%%%%%%%%%%%%%%%%%%%%%%%%%%%%%%%%%%%%%%%%%%%%%%%%%%%%%%%%%%%%%%%%%%%%%%%%%%%%%%%%%%%%%%%%%%%%%%%%%%%%%%%%%%%%%%%%%%%%%%%%%%%%%%%%%%%%%%%%%%%%%%

The second application is on the normalized K\"ahler-Ricci flow on torus fibered minimal models.

Let $(X,\om_0)$ be a compact K\"ahler $(n+m)$-manifold with semiample canonical bundle
and Kodaira dimension $m$. Here we assume $m>0, n>0$. The sections of $K_X^\ell,$ for $\ell$ large, give rise to a fiber space $f:X\to B$ called the {\em Iitaka fibration} of $X$, with $B$ a normal projective variety of dimension $m$ and the smooth fibers $X_b=f^{-1}(b), b\in B\backslash S'$ are all Calabi-Yau $n$-manifolds, diffeomorphic to each other. Let $\chi$ be the restriction of $\frac{1}{\ell}\omega_{FS}$ to $B$, as well as its pullback to $X$. This time we consider the solution $\om=\om(t)$ of the {\em normalized} K\"ahler-Ricci flow
\begin{equation}\label{Equ: normalized KRF}
\frac{\de}{\de t}\omega=-\Ric(\omega)-\omega,\quad\omega(0)=\omega_0.
\end{equation}
which exists for all $t\geq 0$. Thanks to \cite{ST1, ST2, ST3, FZ, TWY, TZ, Gi, GTZ} we have that the evolving metrics have uniformly bounded scalar curvature globally and collapse locally uniformly on $X\backslash S$ to a canonical K\"ahler metric on $B\backslash S'$, and moreover the rescaled metrics along the fibers $e^t\omega|_{X_b}$ converge in $C^\infty$ to a Ricci-flat metric on $X_b$. This is the collapsing phenomenon in the K\"ahler-Ricci flow case.

Now, assume the smooth fibers are the quotient of a complex torus by a holomorphic free action of a finite group, then we have smooth collapsing to the generalized K\"ahler-Einstein metrics defined by Song-Tian \cite{ST2} on the regular part with respect to a fixed metric.

An immediate corollary of Lemma \ref{Thm: BIC principle} is the smooth convergence of the solution and its curvatures.

\begin{theorem}\label{Thm: convergence of KRF with torus fiber}
Let $(X^{n+m},\omega_0)$ with $n>0$ be a compact K\"ahler manifold with $K_X$ semiample and $\kappa(X)=m>0$, and let $f:X\to B$ be the fibration as described above. Assume that for some $y\in B\backslash S'$ the fiber $X_b=f^{-1}(b)$ is bi-holomorphic to a finite quotient of a torus. Let $\om(t),t\in[0,\infty)$ be the solution of the K\"ahler-Ricci flow \eqref{Equ: normalized KRF} starting at $\om_0$. Let $K\subset X\backslash S$ be any compact subset. Then we have
\begin{equation}\label{Equ: convergence of KRF with torus fiber 1}
\|\om(t)-\ti\om(t)\|_{C^k(K,\ti{\om}(t))}\leq h_{K,k}(t),
\end{equation}
where $\ti{\om}(t)=e^{-t}\om_{\SRF}+(1-e^{-t})\om_B$ with $\om_B$ the Song-Tian's generalized K\"ahler-Einstein metric current on $B$. (See section \ref{sec:flow} for its definition)
Moreover, we have the smooth convergence of the curvature tensors
\begin{equation}\label{Equ: convergence Rm curvature of KRF with torus fiber}
\|\Rm(\om(t))-\Rm(\ti{\om}(t))\|_{C^k(K,\ti{\om}(t))}\leq h_{K,k}(t),
\end{equation}
\begin{equation}\label{Equ: convergence Ric curvature of KRF with torus fiber}
\|\Ric(\om(t))-\Ric(\ti{\om}(t))\|_{C^k(K,\ti{\om}(t))}\leq h_{K,k}(t),
\end{equation}
\begin{equation}\label{Equ: convergence scalar curvature of KRF with torus fiber}
\|R(\om(t))-R(\ti{\om}(t))\|_{C^k(K,\ti{\om}(t))}\leq h_{K,k}(t),
\end{equation}
where $h_{K,k}(t)$ are positive functions which tends to zero as $t\to \infty$, depending only on $k$ and the domain $K$. In particular, when $S=\emptyset$, the estimates are globally true and each $h_k(t)$ is of exponential fast decay.
\end{theorem}

Our next result is to exhibit the relation between the Ricci curvature and scalar curvature of the solution $\om(t)$ and the generalized K\"ahler-Einstein metric $\om_B$. We have

\begin{theorem}\label{Thm: convergence Ric curvature to GKE of KRF with torus fiber}
Assume the same set-up as in Theorem \ref{Thm: convergence of KRF with torus fiber}. Let $K\subset X\backslash S$ be any compact subset. Then we have
\begin{equation}\label{Equ: convergence Ric curvature to GKE of KRF with torus fiber}
\|\Ric(\om(t))+\om_B\|_{C^k(K,\ti{\om}(t))}\leq h_{K,k+2}(t),
\end{equation}
and the convergence of scalar curvature
\begin{equation}\label{Equ: convergence scalar curvature to GKE of KRF with torus fiber}
\|R(\om(t))+m\|_{C^k(K,\ti{\om}(t))}\leq h_{K,k+2}(t),
\end{equation}
where $h_{K,k}(t)$ are positive functions which tends to zero as $t\to \infty$, depending only on $k$ and the domain $K$. In particular, when $S=\emptyset$, the estimate is globally true and each $h_k(t)$ is of exponential fast decay.
\end{theorem}

In \cite{J}, the first author showed that the scalar curvature converges to $-m$ in the $C^0_{loc}$ topology in the general fibration case. Here we improved the topology to $C^{\infty}_{loc}$ in the special case when the fibers are flat. We do expect that Theorem \ref{Thm: convergence Ric curvature to GKE of KRF with torus fiber} holds for the general case.\\

%%%%%%%%%%%%%%%%%%%%%%%%%%%%%%%%%%%%%%%%%%%%%%%%%%%%%%%%%%%%%%%%%%%%%%%%%%%%%%%%%%%%%%%%%%%%%%%%%%%%%%%%%%%%%%%%%%%%%%%%%%%

%\subsection{Notations and Conventions} 

This paper is arranged as follows: In section \ref{sec: pf of BIC}, we prove Lemma \ref{Thm: BIC principle} by maximum principle. Then we apply it to Calabi-Yau degenerations in section \ref{sec:CY degeneration}, where Theorem 1.3, 1.4 are proved. Finally in section \ref{sec:flow}, we prove Theorem 1.5 and 1.6 for normalized K\"ahler-Ricci flow on minimal models whose regular fibers are all finite quotients of complex tori.\\

In this paper, we use the following notations and conventions.

We always denote by $h(t)$ a positive function which tends to zero as $t\to \infty$, and by $h_{K,k}(t)$ we mean that this function also depends on the domain $K$ and the order $k$. We allow these functions change from line to line.

When we compute on a product manifold $X=B\times Y$, we always use a product coordinate system and, we call $B$ the base space and the corresponding indices the base directions, and we call $Y$ the fiber space and the corresponding indices the fiber directions.  We will denote any complex $(1,0)$ ``base'' $\mathbb{C}^m$ direction by a subscript $\mathbf{b}$ and any complex $(1,0)$ ``fiber'' $Y$ direction by a subscript $\mathbf{f}$.

By $\nabla^{k,g}$ we means all the possible covariant derivatives with respect to the metric $g$, including holomorphic and anti-holomorphic covariant derivatives when $g$ is a K\"ahler.

{\bf Acknowledgements.} The first author would like to thank his Ph.D advisor Gang Tian for introducing him to geometric analysis, and for his innumerable encouragements and supports. The first author also would like to thank Yashan Zhang and Dongyi Wei for many helpful conversations. The second authors thanks Gang Tian for his constant supports.  Both authors would like to thank Zhenlei Zhang and Jian Song for many helpful discussions.

%%%%%%%%%%%%%%%%%%%%%%%%%%%%%%%%%%%%%%%%%%%%%%%%%%%%%%%%%%%%%%%%%%%%%%%%%%%%%%%%%%%%%%%%%%%%%%%%%%%%%%%%%%%%%%%%%%%%%%%%%%%

\section{Proof of the ``Boundedness Implies Convergence" principle}\label{sec: pf of BIC}

We use the maximum principle to prove Lemma \ref{Thm: BIC principle}.

\begin{proof}[Proof of Lemma \ref{Thm: BIC principle}]
Let $K$ be any compact subset of $U$, and choose compact subsets $\hat{K}\subset\subset U$ and smooth cut-off function $\rho$ as in the Condition $(C)$ of Lemma \ref{Thm: BIC principle}. Let $C_k$ denote constants which depend on $A_1, \dots, A_{k+2}$ which may change from line to line. 

The proof is by induction on the order $k$. The case $k=0$ for Equation \eqref{Equ: convergence in general} is just the Condition $(A)$. Suppose we have established \eqref{Equ: convergence in general} for $0, \dots, k-1$ for $k\geq 1$. Now we prove the estimate \eqref{Equ: convergence in general} for $k$. 

First, using induction hypothesis, we can find some positive function $h(t)$ converging to 0 such that
\begin{equation}\label{Equ: induction hypothesis for k-1}
\left|\nabla^{k-1,\ti{g}(t)}{\eta}\right|_{\ti{g}(t)}^2\leq h(t)
\end{equation}
holds on $\hat{K}$.

Next, using Condition $(B)$, we can compute for every $k\geq 0$ on $U$
\begin{equation}\label{Equ: evolution inequality of eta for general k 1}
\begin{split}
&\left(-\Delta_{\ti{g}(t)}\right)\left(\left|\nabla^{k,\ti{g}(t)}{\eta}\right|_{\ti{g}(t)}^2\right)\\
=& -\ti{g}(t)^{a\bar{b}}\cdot\nabla^{\ti{g}(t)}_a\nabla^{\ti{g}(t)}_{\bar{b}}\left(\left|\nabla^{k,\ti{g}(t)}{\eta}\right|_{\ti{g}(t)}^2\right)\\
=& -2\left|\nabla^{k+1,\ti{g}(t)}{\eta}\right|_{\ti{g}(t)}^2+\nabla^{k+2,\ti{g}(t)}{\eta}*\nabla^{k,\ti{g}(t)}{\eta}\\
\leq& -2\left|\nabla^{k+1,\ti{g}(t)}{\eta}\right|_{\ti{g}(t)}^2+C\cdot\left|\nabla^{k+2,\ti{g}(t)}{\eta}\right|_{\ti{g}(t)}\cdot\left|\nabla^{k,\ti{g}(t)}{\eta}\right|_{\ti{g}(t)}\\
\leq& -2\left|\nabla^{k+1,\ti{g}(t)}{\eta}\right|_{\ti{g}(t)}^2+C\cdot\left|\nabla^{k,\ti{g}(t)}{\eta}\right|_{\ti{g}(t)}
\end{split}
\end{equation}
where $C=C_k$ and $*$ denotes the tensor contraction by the metric $\ti{g}(t)$. Note that in the above inequalities we do not use Bochner-type formulas since we do not have curvature bounds. We use instead the assumption of higher order estimates.  Applying \eqref{Equ: evolution inequality of eta for general k 1} for $k-1\geq 0$, we have on $U$
\begin{equation}\label{Equ: evolution inequality of eta for general k-1}
\left(-\Delta_{\ti{g}(t)}\right)\left(\left|\nabla^{k-1,\ti{g}(t)}{\eta}\right|_{\ti{g}(t)}^2\right)\leq -2\left|\nabla^{k,\ti{g}(t)}{\eta}\right|_{\ti{g}(t)}^2+C\cdot\left|\nabla^{k-1,\ti{g}(t)}{\eta}\right|_{\ti{g}(t)}.
\end{equation}

Also by \eqref{Equ: evolution inequality of eta for general k 1} and Condition $(B)$ , we have on $U$
\begin{equation}\label{Equ: evolution inequality of eta for general k 2}
\begin{split}
&\left(-\Delta_{\ti{g}(t)}\right)\left(\left|\nabla^{k,\ti{g}(t)}{\eta}\right|_{\ti{g}(t)}^4\right)\\
=& 2\left|\nabla^{k,\ti{g}(t)}{\eta}\right|_{\ti{g}(t)}^2\cdot\left(-\Delta_{\ti{g}(t)}\right)\left(\left|\nabla^{k,\ti{g}(t)}{\eta}\right|_{\ti{g}(t)}^2\right)-2\left|\nabla\left|\nabla^{k,\ti{g}(t)}{\eta}\right|_{\ti{g}(t)}^2\right|_{\ti{g}(t)}^2\\
\leq& 2\left|\nabla^{k,\ti{g}(t)}{\eta}\right|_{\ti{g}(t)}^2\cdot\left(-2\left|\nabla^{k+1,\ti{g}(t)}{\eta}\right|_{\ti{g}(t)}^2+C\cdot\left|\nabla^{k,\ti{g}(t)}{\eta}\right|_{\ti{g}(t)}\right)-2\left|\nabla\left|\nabla^{k,\ti{g}(t)}{\eta}\right|_{\ti{g}(t)}^2\right|_{\ti{g}(t)}^2\\
\leq& 2C\cdot\left|\nabla^{k,\ti{g}(t)}{\eta}\right|_{\ti{g}(t)}^3-2\left|\nabla\left|\nabla^{k,\ti{g}(t)}{\eta}\right|_{\ti{g}(t)}^2\right|_{\ti{g}(t)}^2
\end{split}
\end{equation}
where $C=C_k$. Now, take the cut-off function $\rho$ into consideration. By \eqref{Equ: evolution inequality of eta for general k 2} and Condition $(B)$ we have
\begin{equation}\label{Equ: evolution inequality of eta for general k 3}
\begin{split}
&\left(-\Delta_{\ti{g}(t)}\right)\left(\rho^2\left|\nabla^{k,\ti{g}(t)}{\eta}\right|_{\ti{g}(t)}^4\right)\\
=& \rho^2\left(-\Delta_{\ti{g}(t)}\right)\left(\left|\nabla^{k,\ti{g}(t)}{\eta}\right|_{\ti{g}(t)}^4\right)+\left|\nabla^{k,\ti{g}(t)}{\eta}\right|_{\ti{g}(t)}^4\left(-\Delta_{\ti{g}(t)}\right)\left(\rho^2\right)-2\left<\nabla\rho^2, \nabla \left|\nabla^{k,\ti{g}(t)}{\eta}\right|_{\ti{g}(t)}^4\right>_{\ti{g}(t)}\\
\leq& \rho^2\left(C\cdot\left|\nabla^{k,\ti{g}(t)}{\eta}\right|_{\ti{g}(t)}^2-2\left|\nabla\left|\nabla^{k,\ti{g}(t)}{\eta}\right|_{\ti{g}(t)}^2\right|_{\ti{g}(t)}^2\right)+C\cdot \left|\nabla^{k,\ti{g}(t)}{\eta}\right|_{\ti{g}(t)}^4\\
&+C\cdot\rho|\nabla \rho|_{\ti{g}(t)}\left|\nabla^{k,\ti{g}(t)}{\eta}\right|_{\ti{g}(t)}^2\left|\nabla \left|\nabla^{k,\ti{g}(t)}{\eta}\right|_{\ti{g}(t)}^2\right|_{\ti{g}(t)}\\
\leq& C\cdot \left|\nabla^{k,\ti{g}(t)}{\eta}\right|_{\ti{g}(t)}^2+C\cdot \left|\nabla^{k,\ti{g}(t)}{\eta}\right|_{\ti{g}(t)}^4\\
\leq& C\cdot \left|\nabla^{k,\ti{g}(t)}{\eta}\right|_{\ti{g}(t)}^2,\\
\end{split}
\end{equation}
where $C=C_k$. Now, put \eqref{Equ: induction hypothesis for k-1}, \eqref{Equ: evolution inequality of eta for general k-1} and \eqref{Equ: evolution inequality of eta for general k 3} together, we conclude that we can find some $h(t)$ such that on $\hat{K}$
\begin{equation}\label{Equ: inequality for higher order}
\left\{
\begin{aligned}
      &\left|\nabla^{k-1,\ti{g}(t)}{\eta}\right|_{\ti{g}(t)}^2h(t)^{-1}\leq 1,\\
      &\left(-\Delta_{\ti{g}(t)}\right)\left(\left|\nabla^{k-1,\ti{g}(t)}{\eta}\right|_{\ti{g}(t)}^2h(t)^{-1}\right)\leq-2\left|\nabla^{k,\ti{g}(t)}{\eta}\right|_{\ti{g}(t)}^2h(t)^{-1}+1,\\
      &\left(-\Delta_{\ti{g}(t)}\right)\left(\rho^2\left|\nabla^{k,\ti{g}(t)}{\eta}\right|_{\ti{g}(t)}^4h(t)^{-1}\right)\leq C\cdot \left|\nabla^{k,\ti{g}(t)}{\eta}\right|_{\ti{g}(t)}^2h(t)^{-1},\\
\end{aligned}
\right.
\end{equation}
where $C=C_k$. Define an auxiliary function
$$Q:=\rho^2\left|\nabla^{k,\ti{g}(t)}{\eta}\right|_{\ti{g}(t)}^4h(t)^{-1}+C\cdot \left|\nabla^{k-1,\ti{g}(t)}{\eta}\right|_{\ti{g}(t)}^2h(t)^{-1}.$$
Using \eqref{Equ: inequality for higher order}, on $\hat{K}$ we have 
$$\left(-\Delta_{\ti{g}(t)}\right)\left(Q\right)\leq -\left|\nabla^{k,\ti{g}(t)}{\eta}\right|_{\ti{g}(t)}^2h(t)^{-1}+C.$$
Now, at a given time $t\geq 0$, assume $Q$ achieves it's maximum in $U$ at  a point $x_0\in\bar U$. If $x_0\notin \mathrm{Int}\hat K$, then $\rho\equiv 0$, \eqref{Equ: inequality for higher order}, implies that $Q$ has a uniform upper bound $C$, and we are done. Otherwise $x_0\in \mathrm{Int}\hat{K}$, we have 
$$0\leq \left(-\Delta_{\ti{g}(t)}\right)\left(Q\right)(x_0)\leq -\left|\nabla^{k,\ti{g}(t)}{\eta}\right|_{\ti{g}(t)}^2(x_0)h(t)^{-1}+C.$$
which gives 
$$\left|\nabla^{k,\ti{g}(t)}{\eta}\right|_{\ti{g}(t)}^2(x_0)h(t)^{-1}\leq C.$$
Then by Condition $(B)$ and \eqref{Equ: inequality for higher order} we have on $\hat{K}$
$$Q\leq Q(x_0)\leq A_{k}^2\left|\nabla^{k,\ti{g}(t)}{\eta}\right|_{\ti{g}(t)}^2(x_0)h(t)^{-1}+C\cdot \left|\nabla^{k-1,\ti{g}(t)}{\eta}\right|_{\ti{g}(t)}^2(x_0)h(t)^{-1}\leq C.$$
Since $\rho\equiv 1$ on $K$, we obtain
$$\left|\nabla^{k,\ti{g}(t)}{\eta}\right|_{\ti{g}(t)}^2\leq Ch(t)^{\frac{1}{2}}$$
on $K$, where $h(t)$ depends on the constants $A_1, \dots, A_{k+2}, B_K$ and the function $h_0(t)$. This establishes \eqref{Equ: convergence in general} for $k$ and hence completes the proof.
\end{proof}

In applications, it is crucial to have Condition (C). If there is a fibration structure, then we can find such cut-off functions by pulling back a cut-off function from the base manifold, as shown by the following Lemma:

\begin{lemma}\label{Lem:Condition C}
If $f:X^{m+n}\to B$ is a proper holomorphic submersion onto a ball in $\CC^m$, and if $\ti\om(t)$ is of the form $\om_B+e^{-t}\om_0$, where $\om_B$ is a K\"ahler form on $B$ and $\om_0$ is real closed $(1,1)$ form on $X$ whose restriction to each fiber is positive, then for any compact subset $K\subset X$, we can find cut-off function $rho$ satisfying Condition (C) of Lemma \ref{Thm: BIC principle}.
\end{lemma}

\begin{proof}
Since $K$ is compact, so is $f(K)\subset B$. Then we can find a ball $B_1\subset\subset B$ such that $f(K)\subset B_1$. Choose a cut-off function $\rho_0\in C_0^\infty(B)$ such that $\textrm{supp} \rho_0\subset B_1$ and $\rho_0|_{f(K)}\equiv 1$. Then we can take $\rho:=f^*\rho_0$ and $\hat K:=f^{-1}(\overline{B_1})$.
Since
$$\sqrt{-1}\partial\rho_0\wedge\dbar\rho_0\leq C\om_B,\qquad -C\om_B\leq\ddbar\rho_0\leq C\om_B,$$ 
for some $C>0$ on $B$,
we have on $\hat{K}$
\[\begin{split}
&|\nabla\rho|_{\ti{\om}(t)}^2=\sum_{i,j=1}^{m}\ti g(t)^{i\bar{j}}\partial_i\rho\partial_{\bar{j}}\rho\leq C,\\
&|\Delta_{\ti{\om}(t)}\rho|=\left|\tr{\ti{\om}(t)}{\ddbar\rho}\right|\leq C\tr{\ti{\om}(t)}{\om_B}\leq C\\
\end{split}
\]
with some constant $C$ depending on the domain $\hat K$. This verifies Condition $(C)$.
\end{proof}

%%%%%%%%%%%%%%%%%%%%%%%%%%%%%%%%%%%%%%%%%%%%%%%%%%%%%%%%%%%%%%%%%%%%%%%%%%%%%%%%%%%%%%%%%%%%%%%%%%%%%%%%%%%%%%%%%%%%%%%%%%%
%%%%%%%%%%%%%%%%%%%%%%%%%%%%%%%%%%%%%%%%%%%%%%%%%%%%%%%%%%%%%%%%%%%%%%%%%%%%%%%%%%%%%%%%%%%%%%%%%%%%%%%%%%%%%%%%%%%%%%%%%%%

\section{Applications to collapsing Calabi-Yau metrics}\label{sec:CY degeneration}

\subsection{Metric and curvature convergence on locally trivial Calabi-Yau fibrations}

We recall basic definitions in the general fibration case. Let $X$ be a compact K\"ahler $(n+m)$-manifold with $c_1(X)=0$ and let $\om_X$ be a Ricci-flat K\"ahler metric on $X$. Let $f:X\to Z$ be the fibration map with $B=f(X)$. Write $\chi=f^*\omega_Z$, where $\om_Z$ is a smooth K\"ahler form on $Z$. Then $\chi$ is a smooth nonnegative $(1,1)$ form on $X$, and we will also write $\chi$ for the restriction of $\omega_Z$ to $B$.  Note that
$\int_{B\backslash S'}\chi^m$ is finite.

We define a semi Ricci-flat form $\of$ on $X\backslash S$ in the usual way. Namely, for each $b \in B \backslash S'$ there is a smooth function $\rho_b$ on $X_b$ so that $\om_X|_{X_b} + \ddbar \rho_b = \om_{\textrm{SRF,b}}$ is Ricci-flat, normalized by $\int_{X_b} \rho_b (\om_X|_{X_b})^n=0$.  As $b$ varies, this defines a smooth function $\rho$ on $X\backslash S$ and we define
$\of = \om_X + \ddbar \rho.$

Let $F$ be the function on $X\backslash S$ given by
$$F=\frac{\om_X^{n+m}}{\binom{n+m}{n}\of^n\wedge\chi^m}.$$
It is easy to see that $F$ is constant along the fibers $X_b$, $b\in B\backslash S'$, so it descends to a smooth function, also denoted by $F$, on $B\backslash S'$.  We see that $F$ satisfies $\int_{B\backslash S'}F\chi^m=\int_X\om_X^{n+m}/\binom{n+m}{n}\int_{X_y}\om_X^n$ (see \cite[Section 3]{ST2} and \cite[Section 4]{To}). Here note that $\int_{X_b}\om_X^n$ is independent of $b\in B\backslash S'$.

Then \cite[Section 3]{ST2} shows that the Monge-Amp\`ere equation
\begin{equation} \label{Equ: MA of GKE on the base}
(\chi+\ddbar v)^m=\frac{\binom{n+m}{n}\int_{X} \om_X^n\wedge\chi^m}{\int_{X}\om_X^{n+m}}F\chi^m,
\end{equation}
has a unique solution $v$ which is a bounded $\chi$-plurisubharmonic function on $B$, smooth on $B\backslash S'$, with $\int_{X}v\omega_X^{n+m}=0$, where here and henceforth we write $v$ for $\pi^*v$.

Define
$$\om_B=\chi+\ddbar v,$$
for $v$ solving (\ref{Equ: MA of GKE on the base}).  Note that we have
\begin{equation}\label{Equ: ct3}
\of^n\wedge\om_B^m=\frac{\binom{n+m}{n}\int_{X} \om_X^n\wedge\chi^m}{\int_{X}\om_X^{n+m}} F\of^n\wedge\chi^m=\frac{\int_{X} \om_X^n\wedge\chi^m}{\int_{X}\om_X^{n+m}}\om_X^{n+m}.
\end{equation}
Moreover, $\om_B$ is a smooth K\"ahler metric on $B\backslash S'$, and satisfies
$$\Ric(\om_B)=\om_{{\rm WP}},$$
where $\om_{{\rm WP}}$ is the semipositive Weil-Petersson form on $B\backslash S'$, characterizing the change of complex structures of the fibers. If on a domain $U\subset B\backslash S'$ the bundle is holomorphically trivial, then we have $\om_{{\rm WP}}\equiv 0$ on $U$.

For $t\geq 0$, let $\om_t$ be the global reference metrics defined by
$$\om_t=\chi+e^{-t}\omega_X \in \alpha_t = [\chi] + e^{-t} [\om_X],$$ 
 and let $\om(t)=\om_t+\ddbar\vp(t)$ be the unique Ricci-flat K\"ahler metric on $X$ cohomologous to $\om_t$, with the normalization $\int_X\vp\omega_X^{n+m}=0$. Then $\om(t)$ solves the Calabi-Yau equation
\begin{equation}\label{Equ: MA of collapsing CY}
(\om(t))^{n+m}=c_t e^{-nt}\omega_X^{n+m},
\end{equation}
where $c_t$ is the constant given by
\begin{equation}\label{Equ: c_t}
c_t=\frac{\int_X e^{nt}\om_t^{n+m}}{\int_X\om_X^{n+m}}=\frac{1}{\int_X\om_X^{n+m}}\sum_{k=0}^m\binom{n+m}{k}e^{-(m-k)t}\int_X \om_X^{n+m-k}\wedge\chi^k=\binom{n+m}{n} \frac{\int_X\om_X^{n}\wedge\chi^m }{\int_X\om_X^{n+m}}+O(e^{-t}),
\end{equation}
which has a positive limit as $t \to \infty$.

Now, define the reference metrics $\ti{\om}_t$ on the regular part $X\backslash S$ by 
\begin{align}
\ti{\om}_t=\om_B+e^{-t}\of.
\end{align}
In \cite{TWY}, Tosatti-Weinkove-Yang proved the following $C^0$ convergence theorem:
\begin{theorem}[Tosatti-Weinkove-Yang \cite{TWY}]\label{Thm: convergence of TWY}
Let $\om=\om(t) \in \alpha_t$ be Ricci-flat K\"ahler metrics on $X$ as described above.  Then the following holds:
For each compact set $K \subset X \backslash S$,
\begin{equation}\label{Equ: 0-th convergence of TWY}
\|\om(t)-\ti{\om}_t\|_{C^0(K,\ti{\om}_t)} \rightarrow 0, \quad \textrm{as } t\rightarrow \infty.
\end{equation}
In particulr, if S=$\emptyset$, then the convergence is global and exponentially fast.
\end{theorem}

In \cite{HT2}, Hein-Tosatti obtained higher-order estimate  when the smooth fibers are pairwise bi-holomorphic:

\begin{theorem}[Hein-Tosatti \cite{HT2}]\label{Thm: bounded estimate for CY of HT}
Assume that all the fibers  $X_b$ $(b\in U\subset B\setminus f(S))$ are bi-holomorphic to the same Calabi-Yau manifold $Y$. Over any small coordinate ball $U$ compactly contained in $B\setminus f(S)$, use \cite{FG} to trivialize $f$ holomorphically to a product $U\times Y\to U$. Define another Ricci-flat reference K\"ahler forms on $U\times Y$ by $\hat\om(t)=\om_{\mathbb{C}^m}+e^{-t}\om_Y$. Then for any $k \in \N$, there exists a constant $C_{U,k}$ such that
\begin{equation}\label{Equ: bounded estimate for CY of HT}
\|\om(t)\|_{C^k(U\times Y,\hat\om(t))}\leq C_{U,k}.
\end{equation}
holds uniformly for all $t \in [0,\infty)$.
\end{theorem}
Here $\om_Y$ is defined as follows. Let $f:X\to B$ be as in Theorem \ref{Thm: bounded estimate for CY of HT}. By \cite{FG}, $f$ is a holomorphic fiber bundle over $U$. Fix any small coordinate ball in $U$ over which this holomorphic fiber bundle is trivial. We may assume that $U$ is a ball in $\mathbb{C}^m$ and $f:U\times Y\to U$ is the projection, with $Y=X_b$ a compact Calabi-Yau manifold. Then we need to apply a gauge transformation. By \cite[Prop 3.1]{He2} (cf. \cite[Prop 3.1]{GTZ}, \cite[Lemma 4.1]{GW}, \cite[Claim 1, p.382]{He},  \cite[p.2936--2937]{TZ}), there is a unique Ricci-flat K\"ahler metric $\om_Y$ on $Y$ such that, we can find a bi-holomorphism $T$ of $U\times Y$ (over $U$) such that $T^*\om_X=S^*\om_{\mathbb{C}^m}+\om_Y$ for some $S\in \mathrm{GL}(m,\mathbb{C})$. Note that \cite[Prop 3.1]{He2} is stated with $U=\mathbb{C}^m$, but the proof applies also if $U$ is a ball in $\mathbb{C}^m$. Let us also note that $T$ takes the form $T(z,y) =  (z, y + \sigma(z))$, where $\sigma$ is a holomorphic function from $U$ to the space of $g_Y$-parallel $(1,0)$-vector fields on $Y$, and where the addition $y+\sigma(z)$ has the same meaning as in \cite[(1.1)]{He2}. 

We should note that for each $b\in U$ the metrics $\om_{\textrm{SRF,b}}$ and $\om_Y$ are in the same K\"ahler class and both are Ricci-flat, so they are equal by the uniqueness part of Calabi-Yau theorem, i.e., we have
\begin{equation}\label{Equ: identity of reference metric on Y}
\om_{\textrm{SRF,b}}=\om_Y,
\end{equation}
on $Y=X_b$ for all $b\in U$.  

Before we prove Theorem \ref{Thm: convergence of higher order estimate for CY}, we need some useful lemmas.

\begin{lemma}\label{Lem: bounded estimate between reference metric}
Let $(Y,\om_Y)$ be a K\"ahler manifold and $B$ the unit ball in $\mathbb{C}^m$ ($m\geq 1$). Let $\om_1, \om_2$ be any two K\"ahler metrics on $B$, and define two families of product metrics $\hat\om(t)$ and $\ti\om(t)$ for $t\in [0,\infty)$ on $X=B\times Y$ as
\begin{equation}\label{}
\left\{
\begin{aligned}
      &\hat\om(t)=\om_1+e^{-t}\om_Y,\\
      &\ti\om(t)=\om_2+e^{-t}\om_Y,\\
\end{aligned}
\right.
\end{equation}
Then
\begin{equation}\label{Equ: bounded estimate between reference metric}
\|\hat\om(t)\|_{C^k(X,\ti{\om}(t))}\leq C_k.
\end{equation}
for all $k\geq 0$, where the constant $C_k$ depends only on $\om_1$ and $\om_2$, but independent of $t$.
\end{lemma}
\begin{proof}
Denote by $f$ the projection $B\times Y\to B$, and compute under product coordinates. We prove by induction that:
\begin{equation}\label{Equ: gradient of product}
\nabla^{k,\ti{g}(t)}{\hat{g}(t)}=f^*(\alpha_k),
\end{equation}
where $\alpha_k=\nabla^{k,g_1}{g_2}$ is a well-defined covariant tensor on the base space $B$. Indeed, since $\hat\om(t)$ is a product metric, we have
\begin{equation}\label{Equ: Christoffel of product metric hat}
\Gamma(\hat{g}(t))^{k}_{ip}=
\left\{
       \begin{aligned}
       &(g_1)^{k\bar{l}}\nabla^{E}_i(g_1)_{p\bar{l}}, ~~ i,k,p\in \mathbf{b},\\
       &(g_Y)^{k\bar{l}}\nabla^{E}_i(g_Y)_{p\bar{l}}, ~~ i,k,p\in \mathbf{f},\\
       &0, ~~otherwise.\\
       \end{aligned}
\right.
\end{equation}
and similarly
\begin{equation}\label{Equ: Christoffel of product metric tilde}
\Gamma(\ti{g}(t))^{k}_{ip}=
\left\{
       \begin{aligned}
       &(g_2)^{k\bar{l}}\nabla^{E}_i(g_2)_{p\bar{l}}, ~~ i,k,p\in \mathbf{b},\\
       &(g_Y)^{k\bar{l}}\nabla^{E}_i(g_Y)_{p\bar{l}}, ~~ i,k,p\in \mathbf{f},\\
       &0, ~~otherwise.\\
       \end{aligned}
\right.
\end{equation}
So $\hat{g}(t)^{k\bar{l}}\nabla^{\ti{g}(t)}_i\hat{g}(t)_{p\bar{l}}$ is nonzero only if $i,k,p\in \mathbf{b}$, and when $i,k,p\in \mathbf{b}$, we have
\begin{equation}\label{}
\hat{g}(t)^{k\bar{l}}\nabla^{\ti{g}(t)}_i\hat{g}(t)_{p\bar{l}}=(g_1)^{k\bar{l}}\nabla^{E}_i(g_1)_{p\bar{l}}-(g_2)^{k\bar{l}}\nabla^{E}_i(g_2)_{p\bar{l}}=f^*\left((g_1)^{k\bar{l}}\nabla^{g_2}_i(g_1)_{p\bar{l}}\right),
\end{equation}
which gives that
$$\nabla^{\ti{g}(t)}_i\hat{g}(t)_{k\bar{l}}=f^*\left(\nabla^{g_2}_i(g_1)_{k\bar{l}}\right).$$
This is also true for all directions, and hence establishes \eqref{Equ: gradient of product} for $k=1$.

Now assume that we have \eqref{Equ: gradient of product} for $1, 2, \dots, k-1$ with $k\geq 2$. Then we have 
$$\nabla^{k-1,\ti{g}(t)}\hat{g}(t)=f^*\left(\alpha_{k-1}\right),$$
with $\alpha_{k-1}=\nabla^{k-1,g_2}(g_1)$ which is a covariant tensor. Then $f^*\left(\alpha_{k-1}\right)_{i_2, \dots, i_{k+2}}$ is nonzero only if $i_2, \dots, i_{k+2}$ are all of $\mathbf{b}$ or $\overline{\mathbf{b}}$ drections. Now suppose $i_1$ is of the $\mathbf{b}$ or $\mathbf{f}$ directions, then we have
$$\nabla^{\ti{g}(t)}_{i_1}f^*\left(\alpha_{k-1}\right)_{i_2, \dots, i_{k+2}}=\nabla^{E}_{i_1}f^*\left(\alpha_{k-1}\right)_{i_2, \dots, i_{k+2}}-\sum_{2\leq j\leq k+2,i_j\in \mathbf{b}}\Gamma\left(\ti{g}(t)\right)^{p}_{i_1i_j}f^*\left(\alpha_{k-1}\right)_{i_2, \dots, i_{j-1}, p, i_{j+1}, \dots,i_{k+2}}.$$
If $i_1$ is of the $\mathbf{f}$ directions, then since $f^*\left(\alpha_{k-1}\right)_{i_2, \dots, i_{k+2}}$ is a function only depends on the base invariant, and hence $\nabla^{E}_{i_1}f^*\left(\alpha_{k-1}\right)_{i_2, \dots, i_{k+2}}\equiv 0$, and using \eqref{Equ: Christoffel of product metric tilde} we have $\Gamma(\ti{g}(t))^{p}_{i_1i_j}\equiv 0$ since $i_j\in \mathbf{b}$. Hence this covariant derivative is nonzero only if $i_1$ is of the $\mathbf{b}$ directions, and the second terms is nonzero only if $p$ is of the $\mathbf{b}$ directions. Using \eqref{Equ: Christoffel of product metric tilde} again we obtain
$$\nabla^{\ti{g}(t)}_{i_1}f^*\left(\alpha_{k-1}\right)_{i_2, \dots, i_{k+2}}=f^*\left(\nabla^{g_2}_{i_1}\left(\alpha_{k-1}\right)_{i_2, \dots, i_{k+2}}\right).$$
This is also true for all directions. The case for $i_1$ of the $\overline{\mathbf{b}}$ or $\overline{\mathbf{f}}$ directions is similar. Using induction, we obtain \eqref{Equ: gradient of product}.

Now, we have on $X$  the estimate
$$\left|\nabla^{k,\ti{g}(t)}{\hat{g}(t)}\right|_{\ti\om(t)}=\left|f^*(\alpha_k)\right|_{\ti\om(t)}=\left|\alpha_k\right|_{\om_2}\leq C_k.$$
where the constant is independent the time $t$. This completes the proof of the Lemma.
\end{proof}

Next, we compare covariant derivatives of two K\"ahler metrics.

\begin{lemma}\label{Lem: relation between two metrics}
Let $X$ be a K\"ahler manifold. Let $\hat\om, \ti\om$ be any two K\"ahler metrics on $X$ and $\alpha$ be any tensor on $X$. Then we have for any $k\geq 1$
\begin{equation}\label{Equ: relation between three metrics}
\nabla^{k,\ti g}{\alpha}=\sum_{j\geq 1, i_1+\dots +i_j=k,i_1,\dots ,i_j\geq 0}\nabla^{i_1,\hat{g}}{\beta}*\dots*\nabla^{i_j,\hat{g}}{\beta}.
\end{equation}
where $\beta$ means either the metric $\ti{g}$ or the tensor $\alpha$, and $*$ denotes the tensor contraction by $\ti{g}$.
\end{lemma}
\begin{proof}
We prove by induction on $k$. We compute under any given coordinate system $\{z_i\}$ around a given point. 

First consider the $k=1$ case. For example, if $\alpha=\alpha_{k\bar{l}}$ is a two tensor, then we have
\[
\begin{split}
\nabla^{\ti{g}}_i\alpha_{k\bar{l}}-\nabla^{\hat{g}}_i\alpha_{k\bar{l}}
&=\left(\nabla^{E}_i\alpha_{k\bar{l}}-\Gamma(\ti{g})^{p}_{ik}\cdot \alpha_{p\bar{l}}\right)-\left(\nabla^{E}_i\alpha_{k\bar{l}}-\Gamma(\hat{g})^{p}_{ik}\cdot \alpha_{p\bar{l}}\right)\\
&=-\left(\Gamma(\ti{g})^{p}_{ik}-\Gamma(\hat{g})^{p}_{ik}\right)\cdot \alpha_{p\bar{l}}\\
&=-\ti{g}^{p\bar{q}}\cdot\nabla^{\hat{g}}_i\ti{g}_{k\bar{q}}\cdot \alpha_{p\bar{l}}.\\
\end{split}
\]
which gives
$$\nabla^{\ti{g}}\alpha=\nabla^{\hat{g}}\alpha+\nabla^{\hat{g}}\ti{g}*\alpha.$$
where $*$ is tensor contraction by $\ti{g}$. It's easy to see that the same argument holds for any tensor field $\alpha$, this proves \eqref{Equ: relation between three metrics} for $k=1$.

Now assume that we have established \eqref{Equ: relation between three metrics} for $1, 2, \dots, k-1$ with $k\geq 2$. Then we have
\[
\begin{split}
\nabla^{k,\ti{g}}\alpha
&=\nabla^{\ti{g}}\left(\nabla^{k-1,\ti{g}}\alpha\right)\\
&=\nabla^{\hat{g}}\left(\nabla^{k-1,\ti{g}}\alpha\right)+\nabla^{\hat{g}}\ti{g}*\nabla^{k-1,\ti{g}}\alpha\\
&=\left(\nabla^{\hat{g}}+\nabla^{\hat{g}}\ti{g}*\right)\left(\sum_{j\geq 1, i_1+\dots +i_j=k-1,i_1,\dots ,i_j\geq 0}\nabla^{i_1,\hat{g}}{\beta}*\dots*\nabla^{i_j,\hat{g}}{\beta}\right)\\
&=\sum_{j\geq 1, i_1+\dots +i_j=k,i_1,\dots ,i_j\geq 0}\nabla^{i_1,\hat{g}}{\beta}*\dots*\nabla^{i_j,\hat{g}}{\beta},\\
\end{split}
\]
where $\beta$ still denotes either the  metric $\ti{g}$ or the tensor $\alpha$, and $*$ still denotes the tensor contraction by $\ti{g}$. This completes the inductive step and establish this lemma.
\end{proof}

As a corollary, we can change the reference metric in Hein-Tosatti's estimate:
\begin{corollary}\label{Cor: bounded estimate on local base 2}
For all compact sets $K\subset B$ and all $k \in \N$, there exists a constant $C_{K,k}$ independent of $t$ such that for all $t \in [0,\infty)$ we have that
\begin{equation}\label{Equ: bounded estimate 2}
\|\om(t)\|_{C^k(K\times Y,\ti{\om}(t))}\leq C_{K,k}.
\end{equation}
\end{corollary}

\begin{proof}
Denote $\beta(t)$ the K\"ahler form $\om(t)$ or $\ti{\om}(t)$, using Theorem \ref{Thm: bounded estimate for CY of HT} and Lemma \ref{Lem: bounded estimate between reference metric}, we have for any compact subset $K\subset B$
$$\|\beta(t)\|_{C^k(K\times Y,\hat{\om}(t))}\leq C_{K,k}.$$
Hence, using Lemma \ref{Lem: relation between two metrics} with $\alpha=\om(t)$, $\hat{\om}=\hat{\om}(t)$ and $\ti{\om}=\ti{\om}(t)$, we have the following estimate on $K\times Y$ for $k\geq 1$
\[
\begin{split}
&\left|\nabla^{k,\ti{g}(t)}g(t)\right|_{\hat{\om}(t)}\\=
&\left|\sum_{j\geq 1, i_1+\dots +i_j=k,i_1,\dots ,i_j\geq 1}\nabla^{i_1,\hat{g}(t)}{\beta(t)}*\dots*\nabla^{i_j,\hat{g}(t)}{\beta(t)}\right|_{\hat{\om}(t)}\\\leq 
&C\cdot\sum_{j\geq 1, i_1+\dots +i_j=k,i_1,\dots ,i_j\geq 1}\left|\nabla^{i_1,\hat{g}(t)}{\beta(t)}\right|_{\hat{\om}(t)}\cdot \dots \cdot\left|\nabla^{i_j,\hat{g}(t)}{\beta(t)}\right|_{\hat{\om}(t)}\\\leq
&C_{K,k}.\\
\end{split}
\]
where $*$ denotes tensor contraction by $\ti{g}(t)$. Here we have used the uniformly equivalent relations between $\om(t)$, $\hat{\om}(t)$ and $\ti{\om}(t)$, and hence completes the proof of Corollary \ref{Cor: bounded estimate on local base 2}.
\end{proof}

\begin{remark}Suppose we have two uniformly equivalent families of K\"ahler metrics $\om(t)$ and $\ti{\om}(t)$, it doesn't matter which metric we use to measure the norm. Also, assume we have any quantity of the form 
$$A_1*A_2*\dots *A_k,$$
where each $A_i$ is a tensor, and $*$ is the tensor contraction given by $\om(t)$ or $\ti{\om}(t)$, then by the uniformly equivalent relations between $\om(t)$ and $\ti{\om}(t)$, we have
$$\left|A_1*A_2*\dots *A_k\right|\leq C\cdot\left|A_1\right|_{\ti{\om}(t)}\cdot \dots \cdot \left|A_k\right|_{\ti{\om}(t)}.$$
for some uniform constant $C$ independent of $t$, since here we have only finitely many contractions (depending only on k and the degrees of the $A_i$'s). The case for three or more uniformly equivalent metrics is similar. We will use such principle to take norms throughout this paper.
\end{remark}

The following lemma is a standard result in K\"ahler geometry, and follows easily by direct computations. So we omit the proof.
\begin{lemma}\label{Lem: relation of Rm between two metrics}
Given $X$ be a K\"ahler manifold, and $\om$, $\ti{\om}$ be any two K\"ahler forms on $X$. Define the tensor $\Psi$ on $X$ by
$$\Psi^k_{ip}:=\Gamma(g)^{k}_{ip}-\Gamma(\ti{g})^{k}_{ip}=g^{k\bar{l}}\nabla^{\ti{g}}_ig_{p\bar{l}}.$$
Then we have
\begin{equation}\label{Equ: relation of Rm between two metrics 1}
R(\om)_{i\bar{j}k\bar{l}}=\ti{g}^{s\bar{v}}g_{k\bar{v}}R(\ti\om)_{i\bar{j}s\bar{l}}-\nabla^{\ti{g}}_i\nabla^{\ti{g}}_{\bar{j}}g_{k\bar{l}}+g_{u\bar{v}}\Psi^u_{ik}\overline{\Psi^v_{jl}}.
\end{equation}
In particular, we have
\begin{equation}\label{Equ: relation of Rm between two metrics 2}
{R^\sharp(\om)_{i\bar{j}k}}^l={R^\sharp(\ti\om)_{i\bar{j}k}}^l-g^{l\bar{v}}\nabla^{\ti{g}}_{\bar{j}}\nabla^{\ti{g}}_ig_{k\bar{v}}+g^{l\bar{v}}g_{s\bar{t}}\Psi^s_{ik}\overline{\Psi^t_{jv}}.
\end{equation}
\end{lemma}

This lemma shows that we can express the difference of the $\Rm$ curvature tensor of two K\"ahler metrics as the tensor contraction of the first covariant derivatives and second covariant derivatives.

 Now we can prove Theorem \ref{Thm: convergence of higher order estimate for CY}:

\begin{proof}[Proof of Theorem \ref{Thm: convergence of higher order estimate for CY}]
We still just need to verify the Conditions $(A)-(C)$ of Lemma \ref{Thm: BIC principle} with $$\eta(t)=\om(t)-\ti{\om}(t).$$
We have already trivialize $f$ holomorphically to a product $U\times Y\to U$. Let $K\subset\subset U$ be any compact subset.

Condition $(C)$ follows from Lemma \ref{Lem:Condition C}.

Condition $(B)$: Replacing $U$ by a slightly smaller subset. With Theorem \ref{Thm: bounded estimate for CY of HT} at hand, Corollary \ref{Cor: bounded estimate on local base 2} implies the estimate
$$\|\om(t)\|_{C^k(U\times Y,\ti\om(t))}\leq C_{U,k}$$
for any $k\geq 0$.

Condition $(A)$: Replacing $U$ by a smaller subset, we only need to verify the condition 
$$\|\om(t)-\ti\om(t)\|_{C^0(U\times Y,\ti{\om}(t))}\leq h_0(t).$$
From the result of Tosatti-Weinkove-Yang, say Equation \eqref{Equ: 0-th convergence of TWY} of Theorem \ref{Thm: convergence of TWY}, we have
$$\|\om(t)-\ti{\om}_t\|_{C^0(U\times Y,\ti{\om}_t)}^2\leq h(t).$$
where $\ti{\om}_t=\om_B+e^{-t}\of$. Choosing product coordinates, say $\{z_{\alpha}, 1\leq\alpha\leq m; y_i, 1\leq i\leq n\}$ with $z_{\alpha}$ being base coordinates and $y_i$ being fiber coordinates. Then the above estimate implies on $U\times Y$
\begin{equation}\label{Equ: local estimate of metric 1}
\left\{
       \begin{aligned}
       & \left|g(t)_{\alpha\bar{\beta}}-\left[(1-e^{-t})(g_B)_{\alpha\bar{\beta}}+e^{-t}(g_\SRF)_{\alpha\bar{\beta}}\right]\right|^2\leq h(t),~~\alpha, \beta\in \mathbf{b},\\
       & e^t\cdot\left|g(t)_{\alpha\bar{j}}-e^{-t}(g_\SRF)_{\alpha\bar{j}}\right|^2\leq h(t),~~\alpha\in \mathbf{b},j\in \mathbf{f},\\
       & e^{2t}\cdot\left|g(t)_{i\bar{j}}-(g_\SRF)_{i\bar{j}}\right|^2\leq h(t),~~i, j\in \mathbf{f}.\\
       \end{aligned}
\right.
\end{equation}
The first inequality of \eqref{Equ: local estimate of metric 1} implies that
$$\left|g(t)_{\alpha\bar{\beta}}-(g_B)_{\alpha\bar{\beta}}\right|^2\leq 2h(t)+2\left|e^{-t}(g_B)_{\alpha\bar{\beta}}-e^{-t}(g_\SRF)_{\alpha\bar{\beta}}\right|^2\leq h(t),~~\alpha, \beta\in \mathbf{b},$$
and the second inequality of \eqref{Equ: local estimate of metric 1} implies
$$e^t\cdot\left|g(t)_{\alpha\bar{j}}\right|^2\leq 2h(t)+e^{-t}\cdot\left|(g_\SRF)_{\alpha\bar{j}}\right|^2\leq h(t),~~\alpha\in \mathbf{b},j\in \mathbf{f}.$$
Since $\om_{\textrm{SRF,b}}=\om_Y$ for any $b\in U$, we have $(g_\SRF)_{i\bar{j}}=(g_Y)_{i\bar{j}}$ with $i,j\in \mathbf{f}$ under the product coordinates, hence the third inequality of \eqref{Equ: local estimate of metric 1} implies
$$e^{2t}\cdot\left|g(t)_{i\bar{j}}-(g_Y)_{i\bar{j}}\right|^2\leq h(t),~~i, j\in \mathbf{f}.$$
So we conclude that on $U\times Y$ under product coordinates
\begin{equation}\label{Equ: local estimate of metric 2}
\left\{
       \begin{aligned}
       & \left|g(t)_{\alpha\bar{\beta}}-(g_B)_{\alpha\bar{\beta}}\right|^2\leq h(t),~~\alpha, \beta\in \mathbf{b},\\
       & e^t\cdot\left|g(t)_{\alpha\bar{j}}\right|^2\leq h(t),~~\alpha\in \mathbf{b},j\in \mathbf{f},\\
       & e^{2t}\cdot\left|g(t)_{i\bar{j}}-(g_Y)_{i\bar{j}}\right|^2\leq h(t),~~i, j\in \mathbf{f}.\\
       \end{aligned}
\right.
\end{equation}
This implies that
$$\|\om(t)-\ti\om(t)\|_{C^0(U\times Y,\ti{\om}(t))}\leq h(t).$$
This verifies Condition $(A)$ since we already have local uniform equivalence between $\om(t)$ and $\ti{\om}(t)$. 

Now applying Lemma \ref{Thm: BIC principle}, we get the desired estimate \eqref{Equ: convergence of higher order estimate for CY}. 

It remains to prove curvature convergence estimates \eqref{Equ: convergence of Rm in CY case}. 
Applying Lemma \ref{Lem: relation of Rm between two metrics} with $\om=\om(t)$ and $\ti\om=\ti{\om}(t)$, and define the tensor
\begin{equation}\label{Equ: definition of first order derivative 1}
\Psi(t)^k_{ip}:=\Gamma(g(t))^{k}_{ip}-\Gamma(\ti{g}(t))^{k}_{ip}=g(t)^{k\bar{l}}\nabla^{\ti{g}(t)}_ig(t)_{p\bar{l}},
\end{equation}
we get
\[
\begin{split}
    {R^\sharp(\om(t))_{i\bar{j}k}}^l-{R^\sharp(\ti{\om}(t))_{i\bar{j}k}}^l
    &=g(t)^{l\bar{v}}\nabla^{\ti{g}(t)}_{\bar{j}}\nabla^{\ti{g}(t)}_ig(t)_{k\bar{v}}+g(t)^{l\bar{v}}g(t)_{s\bar{t}}\Psi(t)^s_{ik}\overline{\Psi(t)^t_{jv}}\\
    &=g(t)^{l\bar{v}}\nabla^{\ti{g}(t)}_{\bar{j}}\nabla^{\ti{g}(t)}_ig(t)_{k\bar{v}}+g(t)^{l\bar{v}}g(t)^{p\bar{q}}\nabla^{\ti{g}(t)}_ig(t)_{k\bar{q}}\nabla^{\ti{g}(t)}_{\bar{j}}g(t)_{p\bar{v}}.\\
\end{split}
\]
 Then by induction we can show that for all $k\geq 0$
\begin{equation}\label{Equ: higher-order derivatives of Q}
\nabla^{k,\ti{g}(t)}\left[R^\sharp(\om(t))-R^\sharp(\ti{\om}(t))\right]=\sum_{j\geq 1,i_1+\dots +i_j=k+2,i_1,\dots ,i_j\geq 1}\nabla^{i_1,\ti{g}(t)}g(t)*\dots *\nabla^{i_j,\ti{g}(t)}g(t).    
\end{equation}
where $*$ is the tensor contraction given by $g(t)$. In fact, $k=0$ case already follows from \eqref{Equ: relation of Rm between two metrics 2}. Suppose we have \eqref{Equ: higher-order derivatives of Q} for $0, \dots, k-1$ for $k\geq 1$. Then for $k$ we have
\[
\begin{split}
    \nabla^{k,\ti{g}(t)}\left[R^\sharp(\om(t))-R^\sharp(\ti{\om}(t))\right]
    &=\nabla^{\ti{g}(t)}\left\{\sum_{j\geq 1,i_1+\dots +i_j=k+1,i_1,\dots ,i_j\geq 1}\nabla^{i_1,\ti{g}(t)}g(t)*\dots *\nabla^{i_j,\ti{g}(t)}g(t)\right\}\\
    &=\sum_{j\geq 1,i_1+\dots +i_j=k+2,i_1,\dots ,i_j\geq 1}\nabla^{i_1,\ti{g}(t)}g(t)*\dots *\nabla^{i_j,\ti{g}(t)}g(t).\\
\end{split}
\]
This proves \eqref{Equ: higher-order derivatives of Q} for $k$.

Now, taking norms with respect to $\ti{g}(t)$ and using the equivalence of $\om(t)$ and $\ti{\om}(t)$ we have on $K\times Y$
\[
\begin{split}
    \left|\nabla^{k,\ti{g}(t)}\left[R^\sharp(\om(t))-R^\sharp(\ti{\om}(t))\right]\right|_{\ti{\om}(t)}
    &\leq C_k\cdot\sum_{j\geq 1,i_1+\dots +i_j=k+2,i_1,\dots ,i_j\geq 1}\left|\nabla^{i_1,\ti{g}(t)}g(t)\right|_{\ti{\om}(t)}\cdot\dots\cdot\left|\nabla^{i_j,\ti{g}(t)}g(t)\right|_{\ti{\om}(t)}\\
    &\leq h_{K,k}(t).\\
\end{split}
\]
So we get \eqref{Equ: convergence of Rm in CY case}, and this completes the proof.
\end{proof}

\subsection{Metric and curvature convergence on torus-fibered Calabi-Yau manifolds}\label{sec:CY degeneration 2}

Now we consider the case when the smooth fibers $X_b$ are finite quotients of complex tori, but we allow the complex structure to change. We denote the semi-Ricci flat form $\om_{\SRF}$ by $\om_{\SF}$ now, since in this case, its restriction to each smooth fiber is actually flat. We have the following higher-order estimates.
\begin{lemma}[\cite{GTZ, HT, TZ}]\label{Lemma: basic lemma 1}
For all compact sets $K\subset X\backslash S$ and all $k\geq 0$, there exists constants $C_{K,k}$ independent of $t$ such that for all $t\in [0,\infty)$ we have that
\begin{equation}\label{Equ: higher-order bounded 2}
\|\om(t)\|_{C^k(K,\ti\om_t)}\leq C_{K,k}.
\end{equation}
and the curvature bound
\begin{equation}\label{Equ: Rm curvature bounded 2}
\|\Rm(\om(t))\|_{C^k(K,\ti\om_t)}\leq C_{K,k}.
\end{equation}
\end{lemma}

Now, we can prove Theorem \ref{Thm: convergence Rm of CCY with torus fiber}. 
\begin{proof}[Proof of Theorem \ref{Thm: convergence Rm of CCY with torus fiber}]
For all compact sets $K\subset X\backslash S$, we choose disks $B_1\subset\subset B_2\subset\subset B\backslash S'$ such that $K\subset f^{-1}(B_1)$. Set $U=f^{-1}(B_2)$. Then, as the proof of Theorem \ref{Thm: convergence of higher order estimate for CY}, with the help of Theorem \ref{Thm: convergence of TWY} and Lemma \ref{Lemma: basic lemma 1}, we can similarly verify the Conditions $(A)-(C)$ of Lemma \ref{Thm: BIC principle} with $$\eta(t)=\om(t)-\ti\om_t.$$
So we immediately get the higher-order convergence
$$\|\om(t)-\ti\om_t\|_{C^k(K,\ti\om_t)}\le h_{K,k}(t),$$
where $h_{K,k}(t)$ are positive functions which tends to zero as $t\to \infty$, depending only on $k$ and the domain $K$. It remains to show that
$$\|\Rm(\om(t))-\Rm(\ti\om_t)\|_{C^k(K,\ti\om_t)}\leq h_{K,k}(t).$$
Applying Lemma \ref{Lem: relation of Rm between two metrics} with $\om=\om(t)$ and $\ti\om=\ti\om_t$, we get
\begin{equation}\label{Equ: relation of the Rm tensor}
R(\om(t))_{i\bar{j}k\bar{l}}=(\ti{g}_t)^{s\bar{v}}g(t)_{k\bar{v}}R(\ti{\om}_t)_{i\bar{j}s\bar{l}}-\nabla^{\ti{g}_t}_i\nabla^{\ti{g}_t}_{\bar{j}}g(t)_{k\bar{l}}+g(t)_{u\bar{v}}\Psi(t)^u_{ik}\overline{\Psi(t)^v_{jl}}.  
\end{equation}
which gives
\begin{equation}\label{Equ: relation of the Rm tensor 2}
R(\ti{\om}_t)_{i\bar{j}k\bar{l}}=g(t)^{s\bar{v}}(\ti{g}_t)_{k\bar{v}}R(\om(t))_{i\bar{j}s\bar{l}}+g(t)^{s\bar{v}}(\ti{g}_t)_{k\bar{v}}\nabla^{\ti{g}_t}_i\nabla^{\ti{g}_t}_{\bar{j}}g(t)_{s\bar{l}}-g(t)^{s\bar{t}}(\ti{g}_t)_{k\bar{t}}g(t)_{u\bar{v}}\Psi(t)^u_{is}\overline{\Psi(t)^v_{jl}}.
\end{equation}
This is equivalent to 
\begin{equation}\label{Equ: relation of the Rm tensor 3}
R(\ti{\om}_t)_{i\bar{j}k\bar{l}}=g(t)^{s\bar{v}}(\ti{g}_t)_{k\bar{v}}R(\om(t))_{i\bar{j}s\bar{l}}+g(t)^{s\bar{v}}(\ti{g}_t)_{k\bar{v}}\nabla^{\ti{g}_t}_i\nabla^{\ti{g}_t}_{\bar{j}}g(t)_{s\bar{l}}-g(t)^{s\bar{t}}g(t)^{p\bar{q}}(\ti{g}_t)_{k\bar{t}}\nabla^{\ti{g}_t}_ig(t)_{s\bar{q}}\nabla^{\ti{g}_t}_{\bar{j}}g(t)_{p\bar{l}}.  
\end{equation}
Hence we have
\begin{equation}\label{Equ: relation of the Rm tensor 4}
\begin{split}
&R(\ti{\om}_t)_{i\bar{j}k\bar{l}}-R(\om(t))_{i\bar{j}k\bar{l}}\\
=&g(t)^{s\bar{v}}[(\ti{g}_t)_{k\bar{v}}-g(t)_{k\bar{v}}]R(\om(t))_{i\bar{j}s\bar{l}}+g(t)^{s\bar{v}}(\ti{g}_t)_{k\bar{v}}\nabla^{\ti{g}_t}_i\nabla^{\ti{g}_t}_{\bar{j}}g(t)_{s\bar{l}}-g(t)^{s\bar{t}}g(t)^{p\bar{q}}(\ti{g}_t)_{k\bar{t}}\nabla^{\ti{g}_t}_ig(t)_{s\bar{q}}\nabla^{\ti{g}_t}_{\bar{j}}g(t)_{p\bar{l}}.\\
\end{split}  
\end{equation}
Taking norms with respect to $\ti{\om}_t$ and applying \eqref{Equ: convergence of CCY with torus fiber} and Lemma \ref{Lemma: basic lemma 1}, we  have that on $U$
\[
\begin{split}
&\left|\Rm(\ti{\om}_t)-\Rm(\om(t))\right|_{\ti{\om}_t}\\
\leq& C\cdot \left|\ti{\om}_t-\om(t)\right|_{\ti{\om}_t}\cdot \left|\Rm(\om(t)\right|_{\ti{\om}_t}+C\cdot \left|\nabla^{2,\ti{g}_t}g(t)\right|_{\ti{\om}_t}+C\cdot \left|\nabla^{\ti{g}_t}g(t)\right|^2_{\ti{\om}_t}\\
\leq& C\cdot h_K(t)\cdot C_K+C\cdot h_{K}(t)+C\cdot h_{K}(t)^2\\
\leq&h_K(t).\\
\end{split}
\]
Also from \eqref{Equ: relation of the Rm tensor 3}, by induction we can easily obtain that for all $k\geq 0$
\begin{equation}\label{Equ: higher-order derivatives of Rm}
\begin{split}
&\nabla^{k,\ti{g}_t}\Rm(\ti{\om}_t)\\
=&\sum_{j\geq 1,i_1+\dots +i_j=k,i_1,\dots ,i_j\geq 0}\nabla^{i_1,\ti{g}_t}g(t)*\dots *\nabla^{i_{j-1},\ti{g}_t}g(t)*\nabla^{i_j,\ti{g}_t}\Rm(\om(t)),\\
\end{split}    
\end{equation}
where $*$ denotes tensor contraction by $g(t)$ or $\ti{g}_t$. This implies that on $U$
\[
\begin{split}
&\left|\nabla^{k,\ti{g}_t}\Rm(\ti{\om}_t)\right|_{\ti{\om}_t}\\
\leq& C\cdot \sum_{j\geq 1,i_1+\dots +i_j=k,i_1,\dots ,i_{j}\geq 0}\left|\nabla^{i_1,\ti{g}_t}g(t)\right|_{\ti{\om}_t}\cdot\dots\cdot\left|\nabla^{i_{j-1},\ti{g}_t}g(t)\right|_{\ti{\om}(t)}\cdot\left|\nabla^{i_j,\ti{g}_t}\Rm(\om(t))\right|_{\ti{\om}_t}\\
\leq&C_{K,k}.\\
\end{split}
\]
Hence if we set
$$\ti\eta(t)=\Rm(\ti{\om}_t)-\Rm(\om(t)),$$
then we have
\begin{equation}\label{Equ: conditions of Rm 1}
\left\{
\begin{aligned}
      &\|\ti\eta(t)\|_{C^0(U,\ti\om_t)}=\|\Rm(\ti{\om}_t)-\Rm(\om(t))\|_{C^0(U,\ti\om_t)}\leq h_K(t),\\
      &\|\ti\eta(t)\|_{C^k(U,\ti\om_t)}=\|\Rm(\ti{\om}_t)-\Rm(\om(t))\|_{C^k(U,\ti\om_t)}\leq C_{K,k},\\
\end{aligned}
\right.
\end{equation}
Hence Conditions $(A)$ and $(B)$ of Lemma \ref{Thm: BIC principle} are satisfied for $\eta(t)=\ti\eta(t)$ and $\ti{g}(t)=\ti{g}_t$. Hence from Lemma \ref{Thm: BIC principle} we conclude the local convergence
$$\|\Rm(\om(t))-\Rm(\ti\om_t)\|_{C^k(K,\ti\om_t)}\leq h_{K,k}(t).$$
This establish \eqref{Equ: convergence Rm of CCY with torus fiber}.
\end{proof}

%%%%%%%%%%%%%%%%%%%%%%%%%%%%%%%%%%%%%%%%%%%%%%%%%%%%%%%%%%%%%%%%%%%%%%%%%%%%%%%%%%%%%%%%%%%%%%%%%%%%%%%%%%%%%%%%%%%%%%%%%%%%%%%%%%%%%%%%%%%%%%%%%%%%%%%%%
%%%%%%%%%%%%%%%%%%%%%%%%%%%%%%%%%%%%%%%%%%%%%%%%%%%%%%%%%%%%%%%%%%%%%%%%%%%%%%%%%%%%%%%%%%%%%%%%%%%%%%%%%%%%%%%%%%%%%%%%%%%%%%%%%%%%%%%%%%%%%%%%%%%%%%%%%

\section{Applications to normalized K\"ahler-Ricci flow on torus-fibered minimal models}\label{sec:flow}

Let  $(X^{m+n},\omega_0)$ be a compact K\"ahler manifold with semi-ample canonical bundle $K_X$ and assume its Kodaira dimension to be $0<m:=\Kod(X)<m+n$. Then the pluricanonical system $|\ell K_X|$ for sufficiently large $\ell \in \mathbb{Z^+}$ induces the so called ``Iitaka fibration" map
\begin{equation}\label{Equ: canonical map}
f:X \to B \subset \mathbb{C}\mathbb{P}^N:=\mathbb{P}H^0(X,K_X^{\otimes \ell}),
\end{equation}
where $B$ is the canonical model of $X$ with $dimB=m$. Let $S'$ be the singular set of $B$ together with the set of critical values of $f$, and we define $S=f^{-1}(S') \subset X$. 

From \eqref{Equ: canonical map}, we have $f^*\mathcal{O}(1)=K_X^{\otimes \ell}$, hence if we let $\chi=\frac{1}{\ell}\omega_{\mathrm{FS}}$ on $\mathbb{P}H^0(X,K_X^{\otimes \ell})$, we have that $f^*\chi$ (denoted by $\chi$ afterwards) is a smooth semi-positive representative of $-c_1(X)$.  Here, $\omega_{\mathrm{FS}}$ denotes the Fubini-Study metric.  

Given the K\"ahler metric $\om_0$ on $X$, since $X_b:=f^{-1}(b)$ is a Calabi-Yau manifold for each $b\in B\backslash S'$, there exists a unique smooth function $\rho_b$ on $X_b$ with $\int_{X_b}\rho_b\omega_0^{n}=0$, such that $\om_0|_{X_b}+\ddbar\rho_b=:\om_b$ is the unique Ricci-flat K\"ahler metric on $X_b$ in the class $[\om_0|_{X_b}]$.  Moreover, $\rho_b$ depends smoothly on $b$, and so define a global smooth function on $X\backslash S$.  We define
$$\om_{\SRF}=\om_0+\ddbar\rho,$$
which is a closed real $(1,1)$-form on $X\backslash S$, called the ``semi-Ricci flat metric''.

Let $\Om$ be the smooth volume form on $X$ with
\begin{equation}\label{Equ: background volume form}
\ddbar\log\Om=\chi,~\int_X\Om=\binom{m+n}{m}\int_X\om_0^{n}\wedge\chi^m.
\end{equation}
Define a function $F$ on $X\backslash S$ by
\begin{equation}
F:=\frac{\Om}{\binom{m+n}{m}\chi^m\wedge\om_{\SRF}^{n}},
\end{equation}
then $F$ is constant along the fiber $X_b$, $b\in B\backslash S'$, so it descends to a smooth function on $B\backslash S^\prime$.  Then \cite{ST2} showed that the Monge-Amp\'ere equation
\begin{equation}\label{Equ: MAE on canonical model}
(\chi +\ddbar v)^m=Fe^v\chi^m,
\end{equation}
has a unique solution $v\in \PSH(\chi)\cap C^0(B)\cap C^\infty(B\backslash S')$.  Define
$$\om_B=\chi+\ddbar v ,$$
which is a smooth K\"ahler metric on $B\backslash S'$, satisfying the twisted K\"ahler-Einstein equation
$$\Ric(\om_B)=-\om_B+\om_{\mathrm{WP}},$$
where $\om_{\mathrm{WP}}$ is the smooth Weil-Petersson form on $B\backslash S'$.

Now let $\om=\om(t)$ be the solution of the normalized K\"ahler-Ricci flow
\begin{equation}\label{Equ: normalized KRF 2}
\ddt\om=-\Ric(\om)-\om,~\om(0)=\om_0,
\end{equation}
whose solution exists for all time. Define the global reference metrics
$$\hat{\om}(t)=e^{-t}\om_0+(1-e^{-t})\chi,$$
then it is K\"ahler for all $t\ge 0$, and we can write $\om(t)=\hat{\om}(t)+\ddbar\vp(t)$. Then the K\"ahler-Ricci flow \eqref{Equ: normalized KRF 2} is equivalent to the parabolic Monge-Amp\'ere equation
\begin{equation}\label{scalar MAE}
\ddt\vp=\mathrm{log}\frac{e^{(n-m)t}\left(\hat{\om}(t)+\ddbar\vp(t)\right)^n}{\Om}-\vp,~\vp(0)=0.
\end{equation}

We denote by $T_0=\tr{\om(t)}{\om_B}$ and $u=\dot{\vp}+\vp-v$ on $X\backslash S$. Define on $X\backslash S$ the reference metrics
$$\ti{\om}(t)=e^{-t}\om_{\SRF}+(1-e^{-t})\om_B.$$

We always set $K=f^{-1}(K')$， where $K'\subset B\backslash S'$ is a compact subset.  Then we can choose some open subset $U'\subset\subset B\backslash S'$ such that $K'\subset U'$.  Set $U=f^{-1}(U')$, then $K\subset \subset U\subset \subset X\backslash S$.  
First, we have the following lemma in the general fibration case. See \cite{ST2, ST3, FZ, TWY, To4, J}.
\begin{lemma}\label{Lemma: basic lemma 2}
There exist some constant $C=C(K)>0$ and positive functions $h(t)$ which tends to zero as $t\to \infty$, depending on the domain $K$, such that
\begin{enumerate}
\item [(1)]$C^{-1}\hat{\om}(t)\le\om(t)\le C\hat{\om}(t)$, on $K\times[0,\infty).$ 
\item [(2)]$|\vp-v|+|\dot{\vp}+\vp-v|\le h(t),$~on~$K\times[0,\infty).$
\item [(3)]$\|\om(t)-\ti{\om}(t)\|_{C^0(K,\om(t))}\le h(t).$
\item [(4)]$|T_0-m|+\left|\|\om_B\|_{\om(t)}^2-m\right|\le h(t),~on~K\times[0,\infty).$
\item [(5)]There exists a uniform constant $C_0>0$ such that
$$|R|\le C_0,~on~X\times[0,\infty).$$
\item [(6)]Along the normalized K\"ahler-Ricci flow \eqref{Equ: normalized KRF 2}, we have on $X\backslash S$
\begin{equation}\label{Equ: evolution equality of u}
\left(\de_t-\Delta\right)u=\tr{\om(t)}{\om_B}-m.
\end{equation}
\item [(7)]We have
\begin{equation}\label{first derivative convergence2}
|\nabla u|^2\leq h(t),~on~K\times[0,\infty).
\end{equation}
\item [(8)]$\left|R(t)+m\right|\le h(t),~on~K\times[0,\infty).$
\end{enumerate}
Especially, if $S=\emptyset$, then all of the above estimates hold with $K$ replaced by $X$ and $h(t)$ replaced by $Ce^{-\eta t}$ for some constants $\eta,C>0$ depending on $(X,\om_0)$.
\end{lemma}

From now on, assume the smooth fibers are the quotients of complex tori by holomorphic free action of a finite group. In this case, the semi-Ricci flat metric $\om_\SRF$ we constructed above is actually flat when restricted to any smooth fiber $X_b$, $b\in B\backslash S^\prime$, and we denote $\om_\SRF$ by $\om_\SF$ to indicate such semi-flat property. We have the following estimates.

\begin{lemma}[\cite{FZ, Gi0, GTZ, HT, TZ}]\label{Lemma: basic lemma 3}
For all compact sets $K\subset X\backslash S$ and all $k\geq 0$, there exists constants $C_{K,k}$ independent of $t$ such that for all $t\in [0,\infty)$ we have the higher-order derivatives bound
\begin{equation}\label{Equ: higher-order bounded}
\|\om(t)\|_{C^k(K,\ti{\om}(t))}\leq C_{K,k}.
\end{equation}
and the curvatures bound
\begin{equation}\label{Equ: Rm curvature bounded}
\|\Rm(\om(t))\|_{C^k(K,\ti{\om}(t))}\leq C_{K,k}.
\end{equation}
\begin{equation}\label{Equ: Ric curvature bounded}
\|\Ric(\om(t))\|_{C^k(K,\ti{\om}(t))}\leq C_{K,k}.
\end{equation}
\begin{equation}\label{Equ: R curvature bounded}
\|R(\om(t))\|_{C^k(K,\ti{\om}(t))}\leq C_{K,k}.
\end{equation}
\end{lemma}
Now we can prove Theorem \ref{Thm: convergence of KRF with torus fiber}.
\begin{proof}[Proof of Theorem \ref{Thm: convergence of KRF with torus fiber}]
For all compact sets $K\subset X\backslash S$, we choose disks $B_1\subset\subset B_2\subset\subset B\backslash S'$ such that $K\subset f^{-1}(B_1)$. Set $U=f^{-1}(B_2)$. Then, as the proof of Theorem \ref{Thm: convergence Rm of CCY with torus fiber}, with the help of Lemma \ref{Lemma: basic lemma 2} and Lemma \ref{Lemma: basic lemma 3}, we can similarly verify the Conditions $(A)-(C)$ of Lemma \ref{Thm: BIC principle} with $$\eta(t)=\om(t)-\ti{\om}(t).$$
Hence from Lemma \ref{Thm: BIC principle} we have the convergence estimate
$$\|\om(t)-\ti\om(t)\|_{C^k(K,\ti{\om}(t))}\leq h_{K,k}(t).$$
Also, with Lemma \ref{Lemma: basic lemma 1} being replaced by Lemma \ref{Lemma: basic lemma 2}, the same argument as in the proof of Theorem \ref{Thm: convergence Rm of CCY with torus fiber} gives the convergence estimates
\begin{equation}\label{Equ: bound of Rm tensor}
\|\Rm(\om(t))-\Rm(\ti{\om}(t))\|_{C^k(K,\ti{\om}(t))}\leq h_{K,k}(t).
\end{equation}

Next we consider the Ricci curvature. First from \eqref{Equ: bound of Rm tensor} and Lemma \ref{Lemma: basic lemma 3}, we have
\begin{equation}\label{Equ: bound of Rm tensor 2}
\|\Rm(\ti\om(t))\|_{C^k(U,\ti{\om}(t))}\leq C_{K,k}.
\end{equation}
and hence naturally
\begin{equation}\label{Equ: bound of Ric tensor 1}
\|\Ric(\ti\om(t))\|_{C^k(U,\ti{\om}(t))}\leq C_{K,k}.
\end{equation}
Combining \eqref{Equ: bound of Ric tensor 1} and Lemma \ref{Lemma: basic lemma 3} we obtain
\begin{equation}\label{Equ: bound of Ric tensor 3}
\|\Ric(\om(t))\|_{C^k(U,\ti{\om}(t))}+\|\Ric(\ti\om(t))\|_{C^k(U,\ti{\om}(t))}\leq C_{K,k}.
\end{equation}
By definiton, we have
\begin{equation}\label{Equ: relation of Ric tensor}
\begin{split}
    &R(\om(t))_{i\bar{j}}-R(\ti{\om}(t))_{i\bar{j}}\\
    =&g(t)^{k\bar{l}}R(\om(t))_{i\bar{j}k\bar{l}}-\ti{g}(t)^{k\bar{l}}R(\ti{\om}(t))_{i\bar{j}k\bar{l}}\\
    =&\left[g(t)^{k\bar{l}}-\ti{g}(t)^{k\bar{l}}\right]R(\om(t))_{i\bar{j}k\bar{l}}+\ti{g}(t)^{k\bar{l}}\left[R(\om(t))_{i\bar{j}k\bar{l}}-R(\ti{\om}(t))_{i\bar{j}k\bar{l}}\right].\\
\end{split}
\end{equation}
Taking norms with respect to $\ti{\om}(t)$ gives that on $U$
\[
\begin{split}
&\left|\Ric(\om(t))-\Ric(\ti{\om}(t))\right|_{\ti{\om}(t)}\\
\leq& C\cdot \left|\ti{\om}(t)^{-1}-\om(t)^{-1}\right|_{\ti{\om}(t)}\cdot \left|\Rm(\om(t)\right|_{\ti{\om}(t)}+C\cdot \left|\Rm(\om(t))-\Rm(\ti{\om}(t))\right|_{\ti{\om}(t)}\\
\leq& C\cdot h_K(t)\cdot C_{K}+C\cdot h_{K}(t)\\
\leq&h_K(t).\\
\end{split}
\]
Hence the Conditions $(A)-(C)$ of Lemma \ref{Thm: BIC principle} with $$\eta(t)=\Ric(\om(t))-\Ric(\ti{\om}(t))$$
are satisfied, and we conclude from Lemma \ref{Thm: BIC principle} the convergence estimates
\begin{equation}\label{Equ: bound of Ric tensor}
\|\Ric(\om(t))-\Ric(\ti{\om}(t))\|_{C^k(K,\ti{\om}(t))}\leq h_{K,k}(t).
\end{equation}

Finally, for the scalar curvature, from Equation \eqref{Equ: bound of Ric tensor 3} and Lemma \ref{Lemma: basic lemma 3}, we naturally have
\begin{equation}\label{Equ: bound of scalar curvature}
\|R(\om(t))\|_{C^k(U)}+\|R(\ti\om(t))\|_{C^k(U)}\leq C_{K,k}.
\end{equation}
By definition we have
\begin{equation}\label{Equ: relation of scalar curvature}
\begin{split}
    &R(\om(t))-R(\ti{\om}(t))\\
    =&g(t)^{k\bar{l}}R(\om(t))_{k\bar{l}}-\ti{g}(t)^{k\bar{l}}R(\ti{\om}(t))_{k\bar{l}}\\
    =&\left[g(t)^{k\bar{l}}-\ti{g}(t)^{k\bar{l}}\right]R(\om(t))_{k\bar{l}}+\ti{g}(t)^{k\bar{l}}\left[R(\om(t))_{k\bar{l}}-R(\ti{\om}(t))_{k\bar{l}}\right].\\
\end{split}
\end{equation}
As before, we have on $U$
\[
\begin{split}
&\left|R(\om(t))-R(\ti{\om}(t))\right|\\
\leq& C\cdot \left|\ti{\om}(t)^{-1}-\om(t)^{-1}\right|_{\ti{\om}(t)}\cdot \left|\Ric(\om(t))\right|_{\ti{\om}(t)}+C\cdot \left|\Ric(\om(t))-\Ric(\ti{\om}(t))\right|_{\ti{\om}(t)}\\
\leq& C\cdot h_K(t)\cdot C_{K}+C\cdot h_{K}(t)\leq h_K(t).\\
\end{split}
\]
Hence the Conditions $(A)-(C)$ of Lemma \ref{Thm: BIC principle} with $$\eta(t)=R(\om(t))-R(\ti{\om}(t))$$
are satisfied, and we conclude from Lemma \ref{Thm: BIC principle} that
$$\|R(\om(t))-R(\ti{\om}(t))\|_{C^k(K,\ti{\om}(t))}\leq h_{K,k}(t).$$
This completes the proof of Theorem \ref{Thm: convergence of KRF with torus fiber}.
\end{proof}
Finally, we prove Theorem \ref{Thm: convergence Ric curvature to GKE of KRF with torus fiber}. First, we need the following Proposition.

\begin{proposition}\label{Prop: higher-order derivative bound of GKE}
For all compact sets $U\subset\subset X\backslash S$ and all $k\geq 0$, there exists constants $C_{U,k}$ independent of $t$ such that for all $t\in [0,\infty)$ we have that
\begin{equation}\label{Equ: higher-order derivative bound of GKE}
\|\om_B\|_{C^k(U,\ti{\om}(t))}\leq C_{U,k}.
\end{equation}
\end{proposition}
\begin{proof}
We adopt the notations of \cite[Theorem 5.24, p363-368]{To4}. We just need to consider the case when $X_b$ is in fact bi-holomorphic to a torus for some $b\in B\backslash S'$. Then we have
$$\lambda_t^*p^*f^*\om_B=p^*f^*\om_B,$$
and 
$$\lambda_t^*p^*\om_{\SF}=\lambda_t^*\ddbar\eta=\ddbar(\eta\circ\lambda_t)=e^t\ddbar\eta=e^tp^*\om_{\SF}.$$
Hence we have
$$\lambda_t^*p^*(\ti{\om}(t))=\lambda_t^*p^*((1-e^{-t})f^*\om_B+e^{-t}\om_{\SF})=(1-e^{-t})p^*f^*\om_B+p^*\om_{\SF}=:\Check{\om}(t).$$
Now we may assume that $t\geq 1$ so that the factor $(1-e^{-t})\in (\frac{1}{2},1)$ would be a harmless factor. Then we have for every compact set $K\subset B'\times\mathbb{C}^{n-m}$ (here $B'$ is the $B$ in the notations of \cite[Theorem 5.24, p363-368]{To4}, which is a small ball in the regular part of the base space $B$ here) there are constants $C_{K,k}$ such that
$$\|p^*f^*\om_B\|_{C^k(K,\Check{\om}(t))}\leq C_{K,k},$$
for all $t\geq 1$. We can rewrite this as 
$$\|\lambda_t^*p^*f^*\om_B\|_{C^k(K,\lambda_t^*p^*(\ti{\om}(t)))}\leq C_{K,k}.$$
Given any open subset $U\subset X\backslash S$, if $K'\subset\subset U\subset X\backslash S$ is a compact set which is small enough so that $K=p^{-1}(K')\subset B'\times\mathbb{C}^{n-m}$ is compact and $p$ is a bi-holomorphism on $K$, (note that such compact sets $K'$ cover $U$) then we have
$$\|f^*\om_B\|_{C^k(K',\ti{\om}(t))}=\|p^*f^*\om_B\|_{C^k(K,p^*f^*\ti{\om}(t))}=\|\lambda_t^*p^*f^*\om_B\|_{C^k(\lambda_{1/t}(K),\lambda_t^*p^*(\ti{\om}(t)))},$$
where $\lambda_{1/t}$ is the inverse map of $\lambda_t$. But the compact sets $\lambda_{1/t}(K)$ ($t\geq 1$) are all contained in a fixed compact set of $B'\times\mathbb{C}^{n-m}$, hence we have
$$\|f^*\om_B\|_{C^k(K',\ti{\om}(t))}=\|\lambda_t^*p^*f^*\om_B\|_{C^k(\lambda_{1/t}(K),\lambda_t^*p^*(\ti{\om}(t)))}\leq C_{K',k},$$
and so by a covering argument we easily obtain
$$\|\om_B\|_{C^k(U,\ti{\om}(t))}\leq C_{U,k}.$$
This finishes the proof of Proposition \ref{Prop: higher-order derivative bound of GKE}.
\end{proof}

\begin{remark}
The same conclusion of Proposition \ref{Prop: higher-order derivative bound of GKE} holds with $\om_B$ being replaced by any other fixed K\"ahler metric on the regular part of the base space $B$. 
\end{remark}

Now we can prove Theorem \ref{Thm: convergence Ric curvature to GKE of KRF with torus fiber}.
\begin{proof}[Proof of Theorem \ref{Thm: convergence Ric curvature to GKE of KRF with torus fiber}]
Along the normalized K\"ahler-Ricci flow \eqref{Equ: normalized KRF 2}, we have  on $X\backslash S\times[0,\infty)$
\begin{equation}\label{Equ: relation of Ric and u}
\begin{split}
\Ric(\om(t))
&=-\ddbar (\dot{\vp}+\vp)-\chi\\
&=-\ddbar (\dot{\vp}+\vp-v)-(\chi+\ddbar v)\\
&=-\ddbar u-\om_B.\\
\end{split}
\end{equation}
We define $\eta(t)$ to be the $(1,1)$ form
$$\eta(t)=\Ric(\om(t))+\om_B=-\ddbar u.$$
Combining Lemma \ref{Lemma: basic lemma 3} with Proposition \ref{Prop: higher-order derivative bound of GKE}, we have
\begin{equation}\label{Equ: higher-order bound for eta}
\|\eta(t)\|_{C^k(K,\ti{\om}(t))}\leq \|\Ric(\om(t))\|_{C^k(K,\ti{\om}(t))}+\|\om_B\|_{C^k(K,\ti{\om}(t))}\leq C_{K,k},
\end{equation}
for any compact subset $K\subset\subset X\backslash S$.

We need to prove that 
\begin{equation}\label{Equ: general order convergence for eta}
\|\eta(t)\|_{C^0(K,\ti{\om}(t))}\leq h_K(t),
\end{equation}
for any compact subset $K\subset\subset X\backslash S$. To this end, we choose disks $B_1\subset\subset B_2\subset\subset B\backslash S'$ such that $K\subset f^{-1}(B_1)$. Set $U=f^{-1}(B_2)$. According to Lemma \ref{Lemma: basic lemma 2}, we have the estimate  
$$|\nabla u|_{\ti{\om}(t)}^2\leq h_K(t)$$
on $U\times[0,\infty)$.
Combining Equations \eqref{Equ: bound of Rm tensor 2} and \eqref{Equ: higher-order bound for eta} (with $K$ being replaced by $U$ in Equation \eqref{Equ: higher-order bound for eta}), we have on $U\times[0,\infty)$
\[
\begin{split}
&\left(-\Delta_{\ti{\om}(t)}\right)\left(\left|\nabla u\right|_{\ti{\om}(t)}^2\right)\\
=&-2\left|\nabla^{2,\ti{g}(t)}{u}\right|_{\ti{\om}(t)}^2-\ti{g}(t)^{a\bar{b}}\ti{g}(t)^{i\bar{j}}\cdot\nabla^{\ti{g}(t)}_a\nabla^{\ti{g}(t)}_{\bar{b}}\nabla^{\ti{g}(t)}_i{u}\nabla^{\ti{g}(t)}_{\bar{j}}{u}-\ti{g}(t)^{a\bar{b}}\ti{g}(t)^{i\bar{j}}\cdot\nabla^{\ti{g}(t)}_i{u}\nabla^{\ti{g}(t)}_a\nabla^{\ti{g}(t)}_{\bar{b}}\nabla^{\ti{g}(t)}_{\bar{j}}{u}\\
=&-2\left|\nabla^{2,\ti{g}(t)}{u}\right|_{\ti{\om}(t)}^2-\ti{g}(t)^{a\bar{b}}\ti{g}(t)^{i\bar{j}}\cdot\nabla^{\ti{g}(t)}_a\nabla^{\ti{g}(t)}_i\nabla^{\ti{g}(t)}_{\bar{b}}{u}\nabla^{\ti{g}(t)}_{\bar{j}}{u}\\
&-\ti{g}(t)^{a\bar{b}}\ti{g}(t)^{i\bar{j}}\cdot\nabla^{\ti{g}(t)}_i{u}\left\{\nabla^{\ti{g}(t)}_{\bar{b}}\nabla^{\ti{g}(t)}_a\nabla^{\ti{g}(t)}_{\bar{j}}{u}+{\Rm(\ti{\om}(t))^{\bar{q}}}_{\bar{j}a\bar{b}}\nabla^{\ti{g}(t)}_{\bar{q}}{u}\right\}\\
=&-2\left|\nabla^{2,\ti{g}(t)}{u}\right|_{\ti{\om}(t)}^2+\ti{g}(t)^{a\bar{b}}\ti{g}(t)^{i\bar{j}}\cdot\nabla^{\ti{g}(t)}_a\eta_{i\bar{b}}\nabla^{\ti{g}(t)}_{\bar{j}}{u}-\ti{g}(t)^{a\bar{b}}\ti{g}(t)^{i\bar{j}}\cdot\nabla^{\ti{g}(t)}_i{u}\left\{-\nabla^{\ti{g}(t)}_{\bar{b}}\eta_{a\bar{j}}+{\Rm(\ti{\om}(t))^{\bar{q}}}_{\bar{j}a\bar{b}}\nabla^{\ti{g}(t)}_{\bar{q}}{u}\right\}\\
=&-2\left|\nabla^{2,\ti{g}(t)}{u}\right|_{\ti{\om}(t)}^2+\nabla^{\ti{g}(t)}{\eta}*\nabla^{\ti{g}(t)}{u}+\Rm(\ti{\om}(t))*\nabla^{\ti{g}(t)}{u}*\nabla^{\ti{g}(t)}{u}\\
\leq&-2\left|\nabla^{2,\ti{g}(t)}{u}\right|_{\ti{\om}(t)}^2+C\cdot\left|\nabla^{\ti{g}(t)}{\eta}\right|_{\ti{\om}(t)}\cdot\left|\nabla^{\ti{g}(t)}{u}\right|_{\ti{\om}(t)}+C\cdot\left|\Rm(\ti{\om}(t))\right|_{\ti{\om}(t)}\cdot\left|\nabla^{\ti{g}(t)}{u}\right|_{\ti{\om}(t)}^2\\
\leq&-2\left|\nabla^{2,\ti{g}(t)}{u}\right|_{\ti{\om}(t)}^2+C\cdot\left|\nabla^{\ti{g}(t)}{u}\right|_{\ti{\om}(t)}\\
\leq&-2\left|\nabla^{2,\ti{g}(t)}{u}\right|_{\ti{\om}(t)}^2+h(t),\\
\end{split}
\]
where $*$ denotes tensor contraction by $\ti{g}(t)$.  Also, we naturally have
$$\left|\nabla^{2,\ti{g}(t)}{u}\right|_{\ti{\om}(t)}^2\geq \left|\eta\right|_{\ti{\om}(t)}^2.$$
Hence we conclude that
\begin{equation}
\left(-\Delta_{\ti{\om}(t)}\right)\left(\left|\nabla u\right|_{\ti{\om}(t)}^2\right)\leq-2\left|\eta\right|_{\ti{\om}(t)}^2+h(t).\\
\end{equation}
Next, using \eqref{Equ: higher-order bound for eta} we have
\begin{equation}\label{Equ: evolution inequality of eta for k=0}
\begin{split}
&\left(-\Delta_{\ti{\om}(t)}\right)\left(\left|{\eta}\right|_{\ti{\om}(t)}^2\right)\\
=& -2\left|\nabla^{\ti{g}(t)}{\eta}\right|_{\ti{\om}(t)}^2+\nabla^{2,\ti{g}(t)}{\eta}*{\eta}\\
\leq& C\cdot\left|\nabla^{2,\ti{g}(t)}{\eta}\right|_{\ti{\om}(t)}\cdot\left|{\eta}\right|_{\ti{\om}(t)}\\
\leq& C\cdot\left|{\eta}\right|_{\ti{\om}(t)}.\\
\end{split}
\end{equation}
where $*$ denotes tensor contraction by $\ti{g}(t)$. Hence using \eqref{Equ: higher-order bound for eta} we can further compute on $U$
\begin{equation}\label{Equ: evolution inequality of eta for k=0 2}
\begin{split}
&\left(-\Delta_{\ti{\om}(t)}\right)\left(\left|{\eta}\right|_{\ti{\om}(t)}^4\right)\\
\leq& 2\left|{\eta}\right|_{\ti{\om}(t)}^2\cdot\left(-\Delta_{\ti{\om}(t)}\right)\left(\left|{\eta}\right|_{\ti{\om}(t)}^2\right)-2\left|\nabla\left|{\eta}\right|_{\ti{\om}(t)}^2\right|_{\ti{\om}(t)}^2\\
\leq& C\cdot\left|{\eta}\right|_{\ti{\om}(t)}^3-2\left|\nabla\left|{\eta}\right|_{\ti{\om}(t)}^2\right|_{\ti{\om}(t)}^2\\
\leq& C\cdot\left|{\eta}\right|_{\ti{\om}(t)}^2-2\left|\nabla\left|{\eta}\right|_{\ti{\om}(t)}^2\right|_{\ti{\om}(t)}^2.\\
\end{split}
\end{equation}
Choose cut-off function $\rho$ as in the proof of Theorem \ref{Thm: convergence of higher order estimate for CY},  and use again \eqref{Equ: higher-order bound for eta} we can compute on $U$
\[
\begin{split}
&\left(-\Delta_{\ti{\om}(t)}\right)\left(\rho^2\left|{\eta}\right|_{\ti{\om}(t)}^4\right)\\
=& \rho^2\left(-\Delta_{\ti{\om}(t)}\right)\left(\left|{\eta}\right|_{\ti{\om}(t)}^4\right)+\left|{\eta}\right|_{\ti{\om}(t)}^4\left(-\Delta_{\ti{\om}(t)}\right)\left(\rho^2\right)-2\mathrm{Re}\left\{\left<\nabla\rho^2, \nabla \left|{\eta}\right|_{\ti{\om}(t)}^4\right>_{\ti{\om}(t)}\right\}\\
\leq& \rho^2\left(C\cdot\left|{\eta}\right|_{\ti{\om}(t)}^2-2\left|\nabla\left|{\eta}\right|_{\ti{\om}(t)}^2\right|_{\ti{\om}(t)}^2\right)+C\cdot \left|{\eta}\right|_{\ti{\om}(t)}^4+C\cdot\rho|\nabla \rho|_{\ti{\om}(t)}\left|{\eta}\right|_{\ti{\om}(t)}^2\left|\nabla \left|{\eta}\right|_{\ti{\om}(t)}^2\right|_{\ti{\om}(t)}\\
\leq& -2\rho^2\left|\nabla\left|{\eta}\right|_{\ti{\om}(t)}^2\right|_{\ti{\om}(t)}^2+C\cdot \left|{\eta}\right|_{\ti{\om}(t)}^2+C\cdot\left(\rho\left|\nabla \left|{\eta}\right|_{\ti{\om}(t)}^2\right|_{\ti{\om}(t)}\right)\left|{\eta}\right|_{\ti{\om}(t)}^2\\
\leq& -2\rho^2\left|\nabla\left|{\eta}\right|_{\ti{\om}(t)}^2\right|_{\ti{\om}(t)}^2+C\cdot \left|{\eta}\right|_{\ti{\om}(t)}^2+\rho^2\left|\nabla\left|{\eta}\right|_{\ti{\om}(t)}^2\right|_{\ti{\om}(t)}^2+C\cdot \left|{\eta}\right|_{\ti{\om}(t)}^4\\
\leq& C\cdot \left|{\eta}\right|_{\ti{\om}(t)}^2.\\
\end{split}
\]
Now we conclude that we can find some $h(t)$ such that on $U$
\begin{equation}\label{Equ: inequality for higher order 2}
\left\{
\begin{aligned}
      &|\nabla u|_{\ti{\om}(t)}^2h(t)^{-1}\leq 1,\\
      &\left(-\Delta_{\ti{\om}(t)}\right)\left(|\nabla u|_{\ti{\om}(t)}^2h(t)^{-1}\right)\leq-2\left|{\eta}\right|_{\ti{\om}(t)}^2h(t)^{-1}+1,\\
      &\left(-\Delta_{\ti{\om}(t)}\right)\left(\rho^2\left|{\eta}\right|_{\ti{\om}(t)}^4h(t)^{-1}\right)\leq C\cdot \left|{\eta}\right|_{\ti{\om}(t)}^2h(t)^{-1}.\\
\end{aligned}
\right.
\end{equation}
Set
$$Q:=\rho^2\left|{\eta}\right|_{\ti{\om}(t)}^4h(t)^{-1}+C\cdot |\nabla u|_{\ti{\om}(t)}^2h(t)^{-1}.$$
Using \eqref{Equ: inequality for higher order 2}, on $U\times[0,\infty)$ we have 
$$\left(-\Delta_{\ti{\om}(t)}\right)\left(Q\right)\leq -\left|{\eta}\right|_{\ti{\om}(t)}^2h(t)^{-1}+C.$$
Now, at a given time $t$, assume $Q$ achieves it's maximum at point $x_0$. If $x_0\in \de U$, where $\rho\equiv 0$, using \eqref{Equ: inequality for higher order 2}, $Q$ has an upper bound $C$ at time $t$, and we are done. Otherwise $x_0\in U$ and by maximum principle, we have 
$$0\leq \left(-\Delta_{\ti{\om}(t)}\right)\left(Q\right)(x_0)\leq -\left|{\eta}\right|_{\ti{\om}(t)}^2(x_0)h(t)^{-1}+C,$$
which gives 
$$\left|{\eta}\right|_{\ti{\om}(t)}^2(x_0)h(t)^{-1}\leq C.$$
Then by \eqref{Equ: higher-order bound for eta} and \eqref{Equ: inequality for higher order} we have on $U$
$$Q\leq Q(x_0)\leq C_{K}^2\left|{\eta}\right|_{\ti{\om}(t)}^2(x_0)h(t)^{-1}+C\cdot |\nabla u|_{\ti{\om}(t)}^2(x_0)h(t)^{-1}\leq C.$$
Since $\rho\equiv 1$ on $K$, we obtian the estimate
$$\left|{\eta}\right|_{\ti{\om}(t)}^2\leq Ch(t)^{\frac{1}{2}}$$
on $K$. Hence we conclude
\begin{equation}\label{Equ: conditions of Ric}
\left\{
\begin{aligned}
      &\|\eta(t)\|_{C^0(U,\ti\om_t)}\leq h_K(t),\\
      &\|\eta(t)\|_{C^k(U,\ti\om_t)}\leq C_{K,k}.\\
\end{aligned}
\right.
\end{equation}
Now Conditions $(A)$ and $(B)$ of Lemma \ref{Thm: BIC principle} are satisfied for $\eta(t)$. Hence from Lemma \ref{Thm: BIC principle} we conclude the local convergence
\begin{equation}\label{Equ: convergence of Ric}
\|\Ric(\om(t))+\om_B\|_{C^k(K,\ti\om(t))}\leq h_{K,k}(t).
\end{equation}

Finally, for the scalar curvature, Lemma \ref{Lemma: basic lemma 3} gives that
$$\|R(\om(t))\|_{C^k(U,\ti\om(t))}\leq C_{K,k}(t).$$
The estimate \eqref{Equ: convergence of Ric} or the main result of \cite{J} implies that
$$\|R(\om(t))+m\|_{C^0(U)}\leq h(t).$$
Hence set
$$\ti\eta(t)=R(\om(t))+m$$
and apply Lemma \ref{Thm: BIC principle}, we conclude the local convergence
$$\|R(\om(t))+m\|_{C^k(K,\ti\om_t)}\leq h_{K,k}(t).$$
This establish \eqref{Equ: convergence scalar curvature to GKE of KRF with torus fiber} and hence completes the proof of Theorem \ref{Thm: convergence Ric curvature to GKE of KRF with torus fiber}.
\end{proof}

%%%%%%%%%%%%%%%%%%%%%%%%%%%%%%%%%%%%%%%%%%%%%%%%%%%%%%%%%%%%%%%%%%%%%%%%%%%%%%%%%%%%%%%%%%%%%%%%%%%%%%%%%%%%%%%%%%%%%%%%%%%
%%%%%%%%%%%%%%%%%%%%%%%%%%%%%%%%%%%%%%%%%%%%%%%%%%%%%%%%%%%%%%%%%%%%%%%%%%%%%%%%%%%%%%%%%%%%%%%%%%%%%%%%%%%%%%%%%%%%%%%%%%%

\end{document}